\documentclass[a4paper,11pt,twoside]{amsart}

\usepackage{amsfonts,amsmath}
\usepackage{amssymb, latexsym}
\usepackage{amscd,amsthm}
\usepackage{mathrsfs}
\usepackage[all]{xy}
\usepackage[colorlinks=true,citecolor=blue, urlcolor=blue, linkcolor=blue]{hyperref}

\newcommand{\marginlabel}[1]%
 {\mbox{}\marginpar{\toggedleft\hspace{0pt}\bfseries\sf#1}}

\numberwithin{equation}{subsection}

\newtheorem{thm}{Theorem}[section]
\newtheorem{lem}[thm]{Lemma}
\newtheorem{prop}[thm]{Proposition}
\newtheorem{cor}[thm]{Corollary}

\newtheorem{thmInt}{Theorem}[section]

\theoremstyle{definition}
\newtheorem{defn}[thm]{Definition}
\newtheorem{ex}[thm]{Example}
\newtheorem{rem}[thm]{Remark}

\newcommand{\uC}{{\widetilde C}}

\newcommand{\uY}{{\widetilde Y}}

\newcommand{\uPP}{{\widetilde \PP}}

\newcommand{\A}{\mathcal{A}}

\newcommand{\B}{\mathcal{B}}

\newcommand{\C}{\mathcal{C}}

\newcommand{\CC}{\mathbb{C}}

\newcommand{\E}{\mathcal{E}}
\newcommand{\F}{\mathcal{F}}
\newcommand{\G}{\mathcal{G}}
\renewcommand{\H}{\mathcal{H}}
\newcommand{\I}{\mathcal{I}}

\renewcommand{\L}{\mathcal{L}}

\newcommand{\MMM}{\mathfrak{M}}

\newcommand{\N}{\mathcal{N}}
\newcommand{\NN}{\mathbb{N}}
\newcommand{\NNN}{\mathfrak{N}}
\renewcommand{\O}{\mathcal{O}}

\newcommand{\PP}{\mathbb{P}}

\newcommand{\QQ}{\mathbb{Q}}
\newcommand{\RR}{\mathbb{R}}

\newcommand{\ZZ}{\mathbb{Z}}
\renewcommand{\geq}{\geqslant}
\renewcommand{\leq}{\leqslant}

\newcommand{\st}{\;\vline\;} % 'such that' set notation bar
\newcommand{\res}[2]{\left.#1\right|_{#2}} % restriction bar
\newcommand{\cat}[1]{\begin{bf}#1\end{bf}}

\newcommand{\eps}{\varepsilon}

\newcommand{\lra}{\longrightarrow}

\newcommand{\abs}[1]{\left\vert#1\right\vert}
\newcommand{\set}[1]{\left\{#1\right\}}
\newcommand{\gen}[1]{\left\langle#1\right\rangle}

\DeclareMathOperator{\coh}{\cat{Coh}}
\DeclareMathOperator{\Hom}{Hom}
\DeclareMathOperator{\Ext}{Ext}

\DeclareMathOperator{\Bl}{\mathrm{Bl}}

\DeclareMathOperator{\Forg}{Forg}
\renewcommand{\Im}{\mathrm{Im}\,}

\DeclareMathOperator{\coker}{coker}
\DeclareMathOperator{\rk}{rk}
\DeclareMathOperator{\RHom}{RHom}
\DeclareMathOperator{\cExt}{\mathcal{E}\mathit{xt}}
\DeclareMathOperator{\cEnd}{\mathcal{E}\mathit{nd}}

\DeclareMathOperator{\Pic}{Pic}
\DeclareMathOperator{\ch}{ch}

\DeclareMathOperator{\supp}{supp}

\newcommand{\Db}{{\rm D}^{\rm b}}

\newcommand{\HomDB}{\Hom_{\Db(\PP^2,\B_0)}}

\newenvironment{enumerate*}{\begin{enumerate}[topsep=4pt, partopsep=4pt, itemsep=0pt]}{\end{enumerate}}
\newenvironment{itemize*}{\begin{itemize}[topsep=4pt, partopsep=4pt, itemsep=0pt]}{\end{itemize}}
\newenvironment{description*}{\begin{description}[topsep=4pt, partopsep=4pt, itemsep=0pt]}{\end{description}}

\begin{document}

\title[ACM bundles on cubic 3-folds]{ACM bundles on cubic threefolds}

\author[M.~Lahoz, E.~Macr\`i, and P.~Stellari]{Mart\'{\i} Lahoz, Emanuele Macr\`i, and Paolo Stellari}

\address{M.L.: Institut de Math\'{e}matiques de Jussieu -- Paris Rive Gauche (UMR 7586), Universit\'{e} Paris Diderot / Universit\'{e} Pierre et Marie Curie, B\^{a}timent Sophie Germain, Case 7012, 75205 Paris Cedex 13, France}
\email{marti.lahoz@imj-prg.fr}
\urladdr{\url{http://webusers.imj-prg.fr/~marti.lahoz/}}

\address{E.M.: Department of Mathematics, The Ohio State University, 231 W 18th Avenue, Columbus, OH 43210, USA}
\curraddr{Department of Mathematics, Northeastern University, 360 Huntington Avenue, Boston, MA 02115, USA}
\email{e.macri@neu.edu}
\urladdr{\url{http://nuweb15.neu.edu/emacri/}}

\address{P.S.: Dipartimento di Matematica ``F.~Enriques'', Universit{\`a} degli Studi di Milano, Via Cesare Saldini 50, 20133 Milano, Italy}
\email{paolo.stellari@unimi.it}
\urladdr{\url{http://users.unimi.it/stellari}}

\thanks{M.~L.~ is partially supported by SFB/TR 45, Fondation Math\'ematique Jacques Hadamard (FMJH) and MTM2012-38122-C03-02.
E.~M.~ is partially supported by the NSF grants DMS-1001482/DMS-1160466 and DMS-1302730/DMS-1523496, the Hausdorff Center for Mathematics, Universit\"at Bonn, and SFB/TR 45. P.~S.~ is partially supported by the grants FIRB 2012 ``Moduli Spaces and Their Applications'' and
the national research project ``Geometria delle Variet\`a Proiettive'' (PRIN 2010-11).}



\begin{abstract}
	We study ACM bundles on cubic threefolds by using derived category techniques.
	We prove that the moduli space of stable Ulrich bundles of any rank is always non-empty by showing that it is birational to a moduli space of semistable torsion sheaves on the projective plane endowed with the action of a Clifford algebra. We describe this birational isomorphism via wall-crossing in the space of Bridgeland stability conditions, in the example of instanton sheaves of minimal charge.
	\end{abstract}

\maketitle

%%%%%%%%%%%%%%%%%%%%%%%%%%%%%%%%%%%%%

\section*{Introduction}

Fourier--Mukai techniques to study stable vector bundles on surfaces have been an extremely useful tool for more than 30 years.
In this paper, we use a construction by Kuznetsov to generalize such circle of ideas and study Arithmetically Cohen-Macaulay (ACM) stable vector bundles on smooth projective cubic hypersurfaces.
The basic idea is to use a semiorthogonal decomposition of the derived category of coherent sheaves to ``reduce dimension''.
The disadvantage of this approach is that we have to consider complexes and a notion of stability for them; this forces us to restrict to the cubic threefold case (and to special examples in the fourfold case, treated in a forthcoming paper).
The advantage is that this may lead to a general approach to study ACM stable bundles in higher dimensions.

\subsection*{ACM bundles and semiorthogonal decompositions}

Let $Y\subset\PP^{n+1}$ be a smooth complex cubic $n$-fold, and let $\O_Y(H)$ denote the corresponding very ample line bundle.
A vector bundle $F$ on a $Y$ is called \emph{Arithmetically Cohen--Macaulay} if $\dim H^i(Y,F(jH))=0$, for all $i=1,\ldots,n-1$ and all $j\in\ZZ$.
In algebraic geometry, the interest in studying stable ACM bundles (and their moduli spaces) on projective varieties arose from the papers \cite{Beauvillecubic, Druel, Iliev, IM, MT}.
In fact, in \cite{Druel} it is proved that the moduli space of rank $2$ instanton sheaves on a cubic threefold is isomorphic to the blow-up of the intermediate Jacobian in (minus) the Fano surface of lines.
The intermediate Jacobian can be used both  to control the isomorphism type of the cubic, via the Clemens--Griffiths/Tyurin Torelli Theorem, and to prove the non-rationality of the cubic (see \cite{CG}).
From a more algebraic viewpoint, ACM bundles correspond to {\em Maximal Cohen-Macaulay} (MCM) modules over the graded ring associated to the projectively embedded variety, and as such they have been extensively studied in the past years (see, e.g., \cite{Yo}).

In a different direction, Kuznetsov studied in \cite{Kuz:V14} semiorthogonal decompositions of the derived category of a cubic hypersurface.
In fact, as we review in Section \ref{subsec:Semiorth}, there exists a non-trivial triangulated subcategory $\mathbf{T}_Y\subset\Db(Y)$, which might encode the birational information of the cubic.
For example, in the case of a cubic threefold $Y$, it is proven in \cite{BMMS} that the isomorphism class of $Y$ can be recovered directly from $\mathbf{T}_Y$ as a sort of ``categorical version'' of the Clemens--Griffiths/Tyurin Torelli Theorem.
In \cite{Kuz:4fold} it is conjectured that a cubic fourfold is rational if and only if the category $\mathbf{T}_Y$ is equivalent to the derived category of a K3 surface. For the interpretation of $\mathbf{T}_Y$ as a category of matrix factorization, we refer to \cite{OrlovMatrix} while \cite{BB} deals with the interpretation as a summand of the Chow motive of Y.

For cubic threefolds, a different description of $\mathbf{T}_Y$ is available, via Kuznetsov's semiorthogonal decomposition of the derived category of a quadric fibration (see \cite{Kuz:Quadric}).
Indeed, as we review in Section \ref{subsec:quadric}, $\mathbf{T}_Y$ is equivalent to a full subcategory of the derived category of sheaves on $\PP^2$ with the action of a sheaf of Clifford algebras $\B_0$ (determined by fixing a structure of quadric fibration on the cubic).
We denote by $\Xi:\mathbf{T}_Y\hookrightarrow\Db(\PP^2,\B_0)$ the induced fully-faithful functor.
The key observation (which is not surprising if we think to ACM bundles as MCM modules, see \cite[Section 2]{CH1} and \cite{OrlovMatrix}) is the following: given a stable ACM bundle $F$ on $Y$, a certain twist of $F$ by the very ample line bundle $\O_Y(H)$ belongs to $\mathbf{T}_Y$ (this is Lemma \ref{lem:ACM}).
Hence, the idea is to study basic properties of ACM bundles on $Y$ (e.g., existence, irreducibility of the moduli spaces, etc.) by using the functor $\Xi$, and so by considering them as complexes of $\B_0$-modules on $\PP^2$.
The principle is that, since $\Db(\PP^2,\B_0)$ has dimension $2$, although it is not intrinsic to the cubic, it should still lead to several simplifications.
The main question now becomes whether there exists a notion of stability for objects in $\Db(\PP^2,\B_0)$ which corresponds to the usual stability for ACM bundles.
In this paper we suggest that such a notion of stability in $\Db(\PP^2,\B_0)$ should be \emph{Bridgeland stability} \cite{Br}, for cubic threefolds.

\subsection*{Results}

Let $Y$ be a cubic threefold.
By fixing a line $l_0$ in $Y$, the projection from $l_0$ to $\PP^2$ gives a structure of a conic fibration on (a blow-up of) $Y$.
The sheaf of algebras $\B_0$ on $\PP^2$ mentioned before is nothing but the sheaf of even parts of the Clifford algebras associated to this conic fibration (see \cite{Kuz:4fold}).
Denote by $\coh(\PP^2,\B_0)$ the abelian category of coherent $\B_0$-modules, and by $\Db(\PP^2,\B_0)$ the corresponding bounded derived category.

As a first step in the study of ACM bundles on $Y$, we consider the moduli spaces $\MMM_d$ of Gieseker stable $\B_0$-modules in $\coh(\PP^2,\B_0)$ with Chern character $(0,2d,-2d)$, for any $d\geq 1$.
These moduli spaces are tightly related to the geometry of $Y$ and the first general result we can prove is the following (see Theorem \ref{thm:Md}).

\begin{thmInt}\label{thm:mainB0}
	The moduli space $\MMM_d$ is irreducible with a morphism $\varUpsilon: \MMM_d \to \abs{\O_{\PP^2}(d)}$
	whose fiber on a general smooth curve $C$ in $\abs{\O_{\PP^2}(d)}$ is the disjoint union of $2^{5d-1}$ copies of the Jacobian of $C$.
	Moreover, the stable locus $\MMM^s_d$ is smooth of dimension $d^2+1$.
\end{thmInt}

The geometry of $\MMM_1$ and $\MMM_2$ can be understood more explicitly.
Indeed, it turns out that $\MMM_1$ is the Fano variety of lines in $Y$ blown-up at the line $l_0$ (see Proposition \ref{prop:d=1}).
On the other hand $\MMM_2$ is a birational model of the intermediate Jacobian of $Y$ (see Theorem \ref{thm:d=2} for a more detailed statement).
Both results are obtained via wall-crossing in the space of Bridgeland stability conditions on the triangulated category $\Db(\PP^2,\B_0)$. Notice that such a wall-crossing depends on the choice of a line inside the cubic threefold.
As a corollary, one gets that the moduli space of instanton sheaves on $Y$ (of charge $2$) is isomorphic to a moduli space of Bridgeland stable objects in $\Db(\PP^2,\B_0)$ with prescribed Chern character (see Theorem \ref{thm:d=2}).

As $\cat{T}_Y$ can be naturally identified with a full subcategory of $\Db(\PP^2,\B_0)$, via the functor $\Xi$, one may want to consider objects of $\MMM_d$ which are contained in $\cat{T}_Y$.
These generically correspond to ACM bundles on $Y$.
This way we can achieve the following theorem which generalizes one of the main results in \cite{CH}.

\begin{thmInt}\label{thm:main3folds}
	Let $Y$ be a cubic threefold.
	Then, for any $r\geq2$, the moduli space of stable rank $r$ Ulrich bundles is non-empty and smooth of dimension $r^2+1$.
\end{thmInt}

Recall that an \emph{Ulrich bundle} $E$ is an ACM bundle whose graded module $\bigoplus_{m\in\ZZ}H^0(Y,E(m))$ has $3\rk(E)$ generators in degree $1$ (see Section \ref{subsec:Ulrich} for a discussion about the normalization chosen). If compared to the first part of \cite[Thm. 1.2]{CH}, our result removes the genericity assumption.

We believe that Theorem \ref{thm:mainB0} will also be useful in studying the irreducibility of the moduli space of stable Ulrich bundles.
In fact, we expect the functor $\Xi$ to map all stable Ulrich bundles on $Y$ into Bridgeland stable objects in $\Db(\PP^2,\B_0)$, thus generalizing Theorem \ref{thm:d=2} to the case $r>2$.

It is maybe worth pointing out that the proof of Theorem \ref{thm:main3folds}, which is contained in Section \ref{subsec:Ulrich}, is based upon the same deformation argument as in \cite{CH}.
The main difference is that, by using our categorical approach and the moduli spaces $\MMM_d$, we can make it work also for small rank ($r=2,3$).
Indeed, the argument in \cite{CH} relies on the existence of an ACM curve on $Y$ of degree 12 and genus 10, proved by Gei{\ss} and Schreyer in the appendix to \cite{CH}, only for a generic cubic threefold, using {\tt Macaulay2}.
Moreover, although we have focused on cubic threefolds, we believe that our approach might work for any quadric fibration. In particular, other interesting Fano threefolds of Picard rank 1 are the intersection of three quadrics in $\PP^{6}$, the quartic hypersurface containing a double line, or the double covering of $\PP^3$ ramified along a quartic with an ordinary double point (see \cite{be}).

\subsection*{Related works}

The idea of using semiorthogonal decompositions to study ACM bundles by reducing dimension is influenced by \cite{KuzNew}.
More precisely, in \emph{loc. cit.}, Kuznetsov proposes to understand the geometry of moduli spaces of instanton bundles (of any charge) on cubic threefolds via the category $\Db(\PP^2,\B_0)$ and the functor $\Xi$.

There have been many studies about ACM bundles of rank $2$ in dimension $2$ and $3$.
Besides the already mentioned results on instanton bundles on cubic threefolds, some papers in this direction are \cite{AM,BrambillaFaenzi,ChiantiniFaenzi, ChiantiniMadonna2,Madonna2000}.
The higher rank case has been investigated in \cite{ArrondoGrana,AM,Madonna2005}.
The papers \cite{MiroRoigPons} and \cite{PonsLlopisTonini} give few examples of indecomposable ACM bundles of arbitrarily high rank.
The already mentioned papers \cite{CH1,CH} contain a systematic study of stable ACM bundles in higher rank on cubic surfaces and threefolds.
A general existence result for Ulrich bundles on hypersurfaces is in \cite{HUB}.

Regarding preservation of stability via the functor $\Xi$, the papers \cite{BMMS,MS} study the case of ideal sheaves of lines on a cubic threefold.

\subsection*{Plan of the paper}

The paper is organized as follows.
Section \ref{sec:prelim} collects basic facts about semiorthogonal decompositions and general results about ACM bundles on cubic hypersurfaces.
In particular, we show that stable ACM bundles are objects of $\cat{T}_Y$ (up to twists) and state a simple cohomological criterion for a coherent sheaf in $\cat{T}_Y$ to be ACM (see Lemmas \ref{lem:ACM} and \ref{lem:viceversa}).
In Section \ref{subsec:quadric} we review Kuznetsov's work on quadric fibrations.

Section \ref{sec:3folds} concerns the case of cubic threefolds where the first two results mentioned above are proved. The argument is based on a detailed description of the easiest case of $\MMM_1$ which involves Bridgeland stability conditions (see Section \ref{subsec:modMd1}). Some background on the latter subject is provided in the same section. In Sections \ref{subsec:proof} and \ref{subsec:Ulrich} we prove Theorems \ref{thm:mainB0} and \ref{thm:main3folds} respectively. The geometric applications to some simple wall-crossing phenomena are described in detail in Section \ref{sec:modMd2}, where we study the geometry of $\MMM_2$ and its relation to instanton bundles.

\subsection*{Notation}

Throughout this paper we work over the complex numbers.
For a smooth projective variety $X$, we denote by $\Db(X)$ the bounded derived category of coherent sheaves on $X$.
We refer to \cite{huy} for basics on derived categories.
If $X$ is not smooth, we denote by $X_{\mathrm{reg}}$ the regular part of $X$. We set $\hom^i(-,-):=\dim\Hom^i(-,-)$, where $\Hom^i(-,-)$ is computed in an abelian or triangulated category which will be specified each time.
This paper assumes some familiarity with basic constructions and definitions about moduli spaces of stable bundles.
For example, we do not define explicitly the notion of slope and Gieseker stability, of Harder--Narasimhan (HN) and Jordan--H\"{o}lder (JH) factors of a (semistable) vector bundle.
For this, we refer to \cite{HL}.
The same book is our main reference for the standard construction of moduli spaces of stable sheaves.
For the twisted versions of them we refer directly to \cite{simp, Lieb:twisted}.

In the following, we will use the short-hand notation {\em (semi)stable} to refer to {\em stable} (respectively, \emph{semistable}).
Gieseker stability will be simply called stability, while slope stability will be called $\mu$-stability.

\section{The derived category of a cubic hypersurface}\label{sec:prelim}

In this section we show that, on a smooth cubic hypersurface $Y$, all stable ACM bundles are well behaved with respect to Kuznetsov's semiorthogonal decomposition of the derived category.
In particular, after recalling the notion of semiorthogonal decomposition of a derived category, we show that stable ACM bundles on $Y$ belong to the non-trivial component $\cat{T}_Y$ of $\Db(Y)$, up to twist by line bundles.
We also introduce one of the basic tools to study the derived category of cubic threefolds: Kuznetsov's description of the derived category of a quadric fibration.

\subsection{Semiorthogonal decompositions}\label{subsec:Semiorth}
Let $X$ be a smooth projective variety and let $\Db(X)$ be its bounded derived category of coherent sheaves.

\begin{defn}\label{def:semiorth}
	A \emph{semiorthogonal} decomposition of $\Db(X)$ is a sequence of full triangulated subcategories $\cat{T}_1,\ldots,\cat{T}_m\subseteq\Db(X)$ such that $\Hom_{\Db(X)}(\cat{T}_i,\cat{T}_j)=0$, for $i>j$ and, for all $G\in\Db(X)$, there exists a chain of morphisms in $\Db(X)$
	\[
	0=G_m\to G_{m-1}\to\ldots\to G_1\to G_0=G
	\]
	with $\mathrm{cone}(G_i\to G_{i-1})\in\cat{T}_i$, for all $i=1,\ldots,m$.
	We will denote such a decomposition by $\Db(X)=\langle\cat{T}_1,\ldots,\cat{T}_m\rangle$.
\end{defn}

\begin{defn}\label{def:excobjgeneral}
	An object $F\in\Db(X)$ is \emph{exceptional} if $\Hom_{\Db(X)}^p(F,F)=0$, for all $p\neq0$, and $\Hom_{\Db(X)}(F,F)\cong\CC$.
	A collection $\{F_1,\ldots,F_m\}$ of objects in $\Db(X)$ is called an \emph{exceptional collection} if $F_i$ is an exceptional object, for all $i$, and $\Hom_{\Db(X)}^p(F_i,F_j)=0$, for all $p$ and all $i>j$.
\end{defn}

\begin{rem}\label{rmk:exceptional}
	An exceptional collection $\{F_1,\ldots,F_m\}$ in $\Db(X)$ provides a semiorthogonal decomposition
	\[
	\Db(X)=\langle\cat{T},F_1,\ldots,F_m\rangle,
	\]
	where, by abuse of notation, we denoted by $F_i$ the triangulated subcategory generated by $F_i$ (equivalent to the bounded derived category of finite dimensional vector spaces).
	Moreover
	\[
	\cat{T}:=\langle F_1,\ldots,F_m\rangle^\perp=\left\{G\in\Db(X)\,:\,\Hom^p(F_i,G)=0,\text{ for all }p\text{ and }i\right\}.
	\]
	Similarly, one can define ${}^\perp\langle F_1,\ldots,F_m\rangle=\left\{G\in\cat{T}\,:\,\Hom^p(G,F_i)=0,\text{ for all }p\text{ and }i\right\}$.
\end{rem}

Let $F\in \Db(X)$ be an exceptional object.
Consider the two functors, respectively \emph{left and right mutation}, $\cat{L}_F,\cat{R}_F:\Db(X)\to\Db(X)$ defined by
\begin{equation}\label{eqn:LRmutation}
	\begin{split}
	\cat{L}_F(G)&:=\mathrm{cone}\left(\mathrm{ev}:\mathrm{RHom}(F,G)\otimes F\to G\right)\\
	\cat{R}_F(G)&:=\mathrm{cone}\left(\mathrm{ev}^\vee:G\to\mathrm{RHom}(G,F)^\vee\otimes F\right)[-1],
	\end{split}
\end{equation}
where $\mathrm{RHom}(-,-):=\oplus_{p}\Hom_{\Db(X)}^p(-,-)[-p]$.
More intrinsically, let $\iota_{{}^\perp F}$ and $\iota_{F^\perp}$ be the full embeddings of ${}^\perp F$ and $F^\perp$ into $\Db(X)$.
Denote by $\iota^*_{{}^\perp F}$ and $\iota^!_{{}^\perp F}$ the left and right adjoints of $\iota_{{}^\perp F}$ and by $\iota^*_{{F}^\perp}$ and $\iota^!_{{F}^\perp}$ the left and right adjoints of $\iota_{{F}^\perp}$.
Then $\cat{L}_F=\iota_{F^\perp}\circ\iota^*_{F^\perp}$, while $\cat{R}_F=\iota_{{}^\perp F}\circ\iota^!_{{}^\perp F}$ (see, e.g., \cite[Sect.~2]{KuzHPD}).

The main property of mutations is that, given a semiorthogonal decomposition of $\Db(X)$
	\[
	\langle\cat{T}_1,\ldots,\cat{T}_k,F,\cat{T}_{k+1},\ldots,\cat{T}_n\rangle,
	\]
we can produce two new semiorthogonal decompositions
	\[
	\langle\cat{T}_1,\ldots,\cat{T}_k,\cat{L}_F(\cat{T}_{k+1}),F,\cat{T}_{k+2},\ldots,\cat{T}_n\rangle
	\]
and
	\[
	\langle\cat{T}_1,\ldots,\cat{T}_{k-1},F,\cat{R}_F(\cat{T}_k),\cat{T}_{k+1},\ldots,\cat{T}_n\rangle.
	\]

Let us make precise the relation between left and right mutations that will be used throughout the paper.
Denote by $S_X=(-)\otimes\omega_X[\dim(X)]$ the Serre functor of $X$.
We have the following lemma (which actually works more generally for any admissible subcategory in $\Db(X)$).

\begin{lem}\label{lem:adjmutations}
	If $F$ is an exceptional object then $\cat{R}_{S_X(F)}$ is right adjoint to $\cat{L}_{F}$ while $\cat{R}_F$ is left adjoint to $\cat{L}_F$.
\end{lem}

\begin{proof}
	This follows from the remark that ${}^{\perp}(S_X(F))=F^{\perp}$ and using adjunction between the functors $\iota^*_{\cat{D}}$, $\iota_{\cat{D}}$ and $\iota^!_{\cat{D}}$ for $\cat{D}$ equal to ${}^\perp F$ or to $F^\perp$.
\end{proof}

\subsection{ACM bundles on cubics}\label{subsec:ACM}

Let $Y$ be a \emph{smooth cubic $n$-fold}, namely a smooth projective hypersurface of degree $3$ in $\PP^{n+1}$.
We set $\O_Y(H):=\O_{\PP^{n+1}}(H)|_{Y}$.
According to Remark \ref{rmk:exceptional}, as observed by Kuznetsov, the derived category $\Db(Y)$ of coherent sheaves on $Y$ has a semiorthogonal decomposition
\begin{equation}\label{eqn:semiorthgen}
		\Db(Y)=\langle\cat{T}_Y,\O_Y,\O_Y(H),\ldots,\O_Y((n-2)H)\rangle,
\end{equation}
where, by definition,
\begin{equation*}
	\begin{split}
	\cat{T}_Y&:=\langle\O_Y,\ldots,\O_Y(n-2)\rangle^\perp\\
	&=\left\{G\in\Db(Y):\Hom^p_{\Db(Y)}(\O_Y(iH),G)=0,\text{ for all }p\text{ and }i=0,\ldots,n-2\right\}.
	\end{split}
\end{equation*}

Let us first recall the following definition.

\begin{defn}\label{def:ACMbalanced}
	(i) A vector bundle $F$ on a smooth projective variety $X$ of dimension $n$ is \emph{arithmetically Cohen-Macaulay} (ACM) if $\dim H^i(X,F(jH))=0$, for all $i=1,\ldots,n-1$ and all $j\in\ZZ$.
	
	(ii) An ACM bundle $F$ is called \emph{balanced} if $\mu(F)\in[-1,0)$.
\end{defn}

The following lemmas show that the category $\cat{T}_Y$ and stable ACM bundles are closely related.

\begin{lem}\label{lem:ACM}
	Let $Y\subset \PP^{n+1}$ be a smooth cubic $n$-fold.
	Let $F$ be a balanced $\mu$-stable ACM bundle with $\rk(F)>1$.
	Then $F\in \mathbf{T}_Y$.
\end{lem}

\begin{proof}
	We want to show that $h^i(Y,F(-jH))=0$ for all $i\in \ZZ$ and $j\in \set{0,\ldots, n-2}$.
	Since $F$ is ACM, we already have that $h^i(Y,F(-jH))=0$ for $i\in\set{1,\ldots,n-1}$ and any $j$.
	Hence, we only need to prove that $h^0(Y,F(-jH))=h^n(Y,F(-jH))=0$ for $j\in \set{0,\ldots, n-2}$.
	But, on the one hand, we have
	\[
	h^0(Y,F(-jH))=\mathrm{hom}(\O_Y(jH),F)=0,
	\]
	for $j\geq 0$, since $F$ is $\mu$-semistable with $\mu(F)<0$.
	On the other hand,
	\[
	h^n(Y,F(-jH))=\mathrm{ext}^n(\O_Y(jH),F)=\mathrm{hom}(F,\O_Y((-n+1+j)H))= 0,
	\]
	for $-n+1+j< -1$, because $F$ is $\mu$-semistable with $-1\leq\mu(F)$.
	It remains, to prove that the vector space $\Hom(F,\O_Y((-n+1+j)H))$ is trivial for $j=n-2$.
	But this is immediate, since $F$ is a $\mu$-stable sheaf of rank greater than $1$.
\end{proof}

\begin{rem}
	The previous lemma is slightly more general.
	Indeed, the same proof works for a balanced ACM bundle of rank greater than one, if it is $\mu$-semistable and $\Hom(F,\O_Y(-H))=0$.
\end{rem}

\begin{rem}\label{rem:sympl}
When $n=4$, the Serre functor of the subcategory $\cat{T}_Y$ is isomorphic to the shift by $2$ (see \cite[Thm.~4.3]{Kuz:4fold}).
Thus, as an application of the result above and \cite[Thm.~4.3]{KM}, one gets that the smooth locus of any moduli space of $\mu$-stable ACM vector bundles on $Y$ carries a closed symplectic form.
\end{rem}

\begin{lem}\label{lem:viceversa}
	Let $Y\subset \PP^{n+1}$ be a smooth cubic $n$-fold and let $F\in \coh(Y) \cap \cat{T}_Y$.
	Assume
	\begin{equation}\label{eqn:vanishing1}
		\begin{split}
		& H^1(Y,F(H))=0\\
		& H^1(Y,F((1-n)H))=\ldots=H^{n-1}(Y,F((1-n)H))=0.
		\end{split}
	\end{equation}
	Then $F$ is an ACM bundle.
\end{lem}

\begin{proof}
	We start by proving that $H^i(Y,F(jH))=0$, for all $i=1,\ldots,n-1$ and all $j\in\ZZ$.
	Denote by $i:Y\hookrightarrow\PP^{n+1}$ the embedding of $Y$.
	For $m\in\ZZ$, we recall the Beilinson spectral sequence from \cite[Prop.\ 8.28]{huy}:
	\begin{equation*}
		E_1^{p,q}:=H^q(\PP^{n+1}, i_*F(p+m))\otimes \Omega_{\PP^{n+1}}^{-p}(-p)\Rightarrow E^{p+q}=\begin{cases} F(m) &\text{if }p+q=0\\ 0&\text{otherwise.}\end{cases}
	\end{equation*}

	We first consider the case when $m=1$.
	Since $F\in\cat{T}_Y$, we have $E_1^{p,q}=0$, for $p=-n+1,\ldots,-1$ and all $q$.
	By assumption, $E_1^{p,q}=0$, also for $p=-n$ and $q=1,\ldots,n-1$, and $E_1^{0,1}=0$.
	As a consequence, all the differentials $d_r^{-n-1,q}=0$, for all $q=1,\ldots,n-1$ and all $r>0$.
	Hence, $E_1^{-n-1,q}=0$, for all $q=1,\ldots,n-1$.
	A similar argument shows that $E_1^{0,q}=0$, for all $q=1,\ldots,n$.
	Summing up, we have the following vanishing:
	\begin{align}
		& H^i(Y,F(jH))=0, \qquad \text{ for all } i=1,\ldots,n-1 \text{ and all } j=-n,\ldots,0,1\label{eqn:middle}\\
		& H^0(Y,F(jH))=0, \qquad \text{ for all } j\leq0\label{eqn:bottom}\\
		& H^n(Y,F(jH))=0, \qquad \text{ for all } j\geq-n+2.\label{eqn:top}
	\end{align}

	Now, on the one hand, by using the Beilinson spectral sequence for $m>1$ and the vanishing \eqref{eqn:middle} and \eqref{eqn:bottom}, we can prove by induction on $m$ that
	\begin{equation}\label{eqn:left}
		H^i(Y,F(jH))=0, \qquad \text{ for all } i=1,\ldots,n-1 \text{ and all } j\leq0.
	\end{equation}
	On the other hand, for $m<1$, the vanishing \eqref{eqn:middle} and \eqref{eqn:top} show
	\[
	H^i(Y,F(jH))=0, \qquad \text{ for all } i=1,\ldots,n-1 \text{ and all } j>0.
	\]

	To finish the proof of the lemma, we only need to show that $F$ is locally-free.
	Hence, it is enough to prove that $\cExt^i(F,\O_Y)=0$, for all $i>0$.
	For $k\in\ZZ$, consider the local-to-global spectral sequence
	\begin{equation}\label{eqn:splocglob}
		E_2^{p,q}(k)= H^p(Y,\cExt^q(F(k),\O_Y))\Rightarrow \Ext^{p+q}(F(k),\O_Y).
	\end{equation}
	Assume, for a contradiction, that $\cExt^i(F,\O_Y)\neq 0$ for some $i>0$.
	Then, for $k\ll 0$, $E_2^{0,i}(k)\neq 0$, while $E_2^{p,i}(k)= 0$ for $p>0$.
	From the spectral sequence and Serre duality, we deduce that $H^{n-i}(Y,F((k-n+1)H))\neq0$, for $i=1,\ldots,n$ and for $k\ll 0$, contradicting \eqref{eqn:left} and \eqref{eqn:bottom}.
\end{proof}

Finally, for later use, we recall how to construct autoequivalences of $\cat{T}_Y$ (not fixing $\coh(Y)\cap\cat{T}_Y$).

\begin{lem}\label{lem:autoequiv}
	Let $Y\subset \PP^{n+1}$ be a smooth cubic $n$-fold.
	Then, the functor
	\[
	\Theta:\cat{T}_Y\to\cat{T}_Y, \qquad F\mapsto \cat{L}_{\O_Y}\left(F\otimes\O_Y(H) \right)
	\]
	is an autoequivalence of $\cat{T}_Y$.
\end{lem}

\begin{proof}
	Clearly $\cat{L}_{\O_Y}\left(F\otimes\O_Y(H) \right)$ belongs to $\cat{T}_Y$ and the inverse of $\Theta$ is given by the exact functor $\Theta^{-1}(-):=\O_Y(-H)\otimes\cat{R}_{\O_Y}(-)$.
\end{proof}

Let us revise a classical example under a slightly different perspective.

\begin{ex}\label{ex:BBR}
	Denote by $\MMM_Y(v)$ the moduli space of $\mu$-stable torsion-free sheaves $E$ on $Y$ with $v:=\ch(E)=(2,-H,-l/2,1/2)$, and let $F(Y)$ be the Fano surface of lines contained in a cubic threefold $Y$.
	In \cite[Thm. 1]{BBR} it is proven that there exists a connected component $\MMM_Y'(v)\subset \MMM_Y(v)$ consisting of $\mu$-stable ACM bundles such that $\MMM_Y'(v)\cong F(Y)$.

	This can be obtained by using a slightly different approach.
	First of all, observe that the ideal sheaves of lines $\I_l\in\cat{T}_Y$, for all $l\subset Y$ and $F(Y)$ is the moduli space of these sheaves.
	By applying $\Theta[-1]$ (see Lemma \ref{lem:autoequiv}), we get an exact sequence in $\coh(Y)$
	\begin{equation}\label{eqn:Fl}
		0\to F_l:=\Theta(\I_l)[-1] \to \O_Y\otimes_{\CC}H^0(Y,\I_l(H))\to \I_l(H)\to 0.
	\end{equation}
	In particular, all $F_l$ are torsion-free sheaves with Chern character $v$.
	By \eqref{eqn:Fl} and Lemma \ref{lem:viceversa}, we deduce that $F_l$ are all ACM bundles.
	Since they belong to $\cat{T}_Y$, we have $H^0(Y,F_l)=0$ and as they have rank $2$, this shows that they are $\mu$-stable.
	By construction, the Fano variety of lines is then a connected component of $\MMM_Y(v)$.

	By \cite[Lemma 1]{BBR}, the connected component $\MMM_Y'(v)$ can also be characterized as the component of $\MMM_Y(v)$ consisting of $\mu$-stable torsion-free sheaves $G$ satisfying $H^0(Y,G(H))\neq0$ and $H^0(Y,G)=0$.
	Also, by Lemma \ref{lem:ACM} and by \cite[Thm.~4.1 \& Prop.~4.2]{BMMS}, all balanced $\mu$-stable ACM bundles $G$ with $\Ext^1(G,G)\cong\CC^2$ are in $\MMM_Y'(v)$.
\end{ex}

\begin{rem}\label{rmk:Serreconstr}
	Let us remark that the original proof in \cite{BBR} of the result in Example \ref{ex:BBR} relies on the so called \emph{Serre's construction} which we briefly recall it in a more general form (e.g.~\cite{arr}).
	Let $X$ be a smooth projective manifold of dimension at least $3$ and let $E$ be a rank $r$ vector bundle on $X$ which is spanned by its global sections.
	The dependency locus of $r-1$ general sections $s_1,\ldots,s_{r-1}$ of $E$ is a locally complete intersection subscheme $V$ of codimension $2$ in $X$.
	If $L=\det(E)$, then the twisted canonical bundle $K_V\otimes L^{-1}$ is generated by $r-1$ sections.

	Conversely, let $V$ a codimension $2$ locally complete intersection subscheme of $X$ and let $L$ be a line bundle on $X$ such that $H^2(X,L^{-1})=0$.
	If $K_V\otimes L^{-1}$ is generated by $r - 1$ global sections, then $V$ can be obtained as the dependency locus of $r-1$ sections of $E$.
	
	This construction is pervasive in the literature and it has been extensively used in various works to produce examples of stable ACM bundles.
\end{rem}

\subsection{Quadric fibrations}\label{subsec:quadric}

The results of \cite{Kuz:Quadric} on the structure of the derived category of coherent sheaves on a fibration in quadrics will be the basic tools to study the derived category of cubic threefolds.
We briefly summarize them here.

Consider a smooth algebraic variety $S$ and a vector bundle $E$ of rank $n$ on $S$.
We consider the projectivization $q:\PP_S(E)\to S$ of $E$ on $S$ endowed with the line bundle $\O_{\PP_S(E)/S}(1)$.
Given a line bundle $L$ on $S$ and an inclusion of vector bundles $\sigma:L\to\mathrm{Sym}^2 E^\vee$, we denote by $\alpha:X\hookrightarrow\PP_S(E)$ the zero locus of $\sigma$ and by $\pi:X\to S$ the restriction of $q$ to $X$.
It is not difficult to prove that $\pi$ is a flat quadric fibration of relative dimension $n-2$.
The geometric picture can be summarized by the following diagram
\[
\xymatrix{
X \ar@{^{(}->}[rr]^-{\alpha}\ar[rrd]_-{\pi}&& \PP_{S}(E)\ar[d]_-{q}\\
&& S.
}
\]

The quadric fibration $\pi:X\to S$ carries a \emph{sheaf $\B_\sigma$ of Clifford algebras}.
In fact, $\B_\sigma$ is the relative sheafified version of the classical Clifford algebra associated to a quadric on a vector space (more details can be found in \cite[Sect.\ 3]{Kuz:Quadric}).
As in the absolute case, $\B_\sigma$ has an \emph{even part} $\B_0$ whose description as an $\O_S$-module is as follows
\[
\B_0\cong\O_S\oplus(\wedge^2 E\otimes L)\oplus(\wedge^4 E\otimes L^2)\oplus\ldots
\]
The \emph{odd part} $\B_1$ of $\B_\sigma$ is such that
\[
\B_1\cong E\oplus(\wedge^3 E\otimes L)\oplus(\wedge^5 E\otimes L^2)\oplus\ldots
\]
We also denote $\B_{2i}=\B_0\otimes L^{-i}$ and $\B_{2i+1}=\B_1\otimes L^{-i}$.

We write $\coh(S,\B_0)$ for the abelian category of coherent $\B_0$-modules on $S$ and $\Db(S,\B_0)$ for its derived category.

\begin{thm}[{\cite[Thm.~4.2]{Kuz:Quadric}}]\label{thm:KuzQuadric}
	If $\pi:X\to S$ is a quadric fibration as above, then there exists a
	semiorthogonal decomposition
	\[
	 \Db(X)=\langle\Db(S,\B_0),\pi^*(\Db(S))\otimes\O_{X/S}(1),\pi^*(\Db(S))\otimes\O_{X/S}(2),\ldots,\pi^*(\Db(S))\otimes\O_{X/S}(n-2)\rangle,
	\]
	where $\Db(S,\B_0)$ is the derived category of coherent sheaves of $\B_0$-modules on $S$.
\end{thm}

In order to make this result precise, we need to give the definition of the fully faithful functor $\Db(S,\B_0)\to\Db(X)$ providing the embedding in the above semiorthogonal decomposition.
The exact functor $\Phi:=\Phi_{\E'}:\Db(S,\B_0)\to\Db(X)$ is defined as the Fourier--Mukai transform
\begin{equation*}\label{eqn:defemb}
	\Phi_{\E'}(-):=\pi^*(-)\otimes_{\pi^*\B_0}\E',
\end{equation*}
where $\E'\in\coh(X)$ is a rank $2^{n-2}$ vector bundle on $X$ with a natural structure of flat left $\pi^*\B_0$-module defined by the short exact sequence
\begin{equation}\label{eqn:defE'}
	0\longrightarrow q^*\B_0\otimes\O_{\PP_S(E)/S}(-2)\longrightarrow q^*\B_1\otimes\O_{\PP_S(E)/S}(-1)\longrightarrow\alpha_*\E'\longrightarrow
0.
\end{equation}
In the notation of \cite[Lemma 4.5]{Kuz:Quadric}, $\E'=\E'_{-1,1}$.
The left adjoint functor of $\Phi$ is
\begin{equation}\label{eqn:bascoh}
	\Psi(-):=\pi_*((-)\otimes_{\O_X}\E\otimes_{\O_X}\det E^\vee[n-2]),\\
\end{equation}
where $\E\in\coh(X)$ is another rank $2^{n-2}$ vector bundle with a natural structure of right $\pi^*\B_0$-module (see again \cite[Sect.\ 4]{Kuz:Quadric}).
The analogous presentation of $\E$ is
\begin{equation}\label{eqn:defE}
	0\longrightarrow q^*\B_{-1}\otimes\O_{\PP_S(E)/S}(-2)\longrightarrow q^*\B_0\otimes\O_{\PP_S(E)/S}(-1)\longrightarrow\alpha_*\E\longrightarrow
0.
\end{equation}
In the notation of \cite[Lemma 4.5]{Kuz:Quadric}, $\E=\E_{-1,0}$.

\medskip

The category of $\B_0$-modules may be hard to work with directly.
In some cases, we can reduce to a category of modules over a sheaf of Azumaya algebras, which is easier to deal with.
We conclude this section by recalling this interpretation (see \cite[Sections 3.5 \& 3.6]{Kuz:Quadric}).
We define $S_1\subset S$ to be the degeneracy locus of $\pi$, namely the subscheme parameterizing singular quadrics, and $S_2\subset S_1$ the locus of singular quadrics of corank $\geq 2$.
There are two separate cases to consider, according to parity of $n$.

In this paper we just need to study the case when $n$ is odd. To this end, let $f:\widehat{S} \to S$ be the the stack of $2^{\rm nd}$ roots of $\O_S(S_1)$ along the section $S_1$.
An object of this stack over $T\to S$ is a triple $(L,\phi,\delta)$, where $L$ is a line bundle over $T$, $\phi$ is an isomorphism of $L^2$ with the pullback of $\O_{S}(S_1)$ to $T$ and $\delta$ is a section of $L$ such that $\phi(\delta^2)=S_1$ (see \cite{AGV,Ca}).
Locally over $S$, the category of coherent sheaves on $\widehat{S}$ can be identified with the category of coherent sheaves on the double covering of $S$ ramified along $S_1$
which are $\ZZ/2\ZZ$-equivariant with respect to the involution of the double covering (which only exists locally).
That is, another way of saying, the category of coherent sheaves on the quotient stack of the double cover by the involution.
Kuznestov calls the noncommutative variety $\widehat{S}$, ``$S$ with a $\ZZ/2\ZZ$-stack structure along $S_1$'' (see \cite[Ex.~2.2]{Kuz:Quadric}).

\begin{prop}[{\cite[Prop.~3.15]{Kuz:Quadric}}]\label{prop:AzumayaStack}
There exists a sheaf of algebras $\A_0$ on $\widehat{S}$ such that $f_*\A_0=\B_0$ and
\[
f_*:\coh(\widehat{S},\A_0)\xrightarrow{\sim}\coh(S,\B_0)
\]
is an equivalence of categories.
Moreover, the restriction of $\A_0$ to the complement of $\widehat{S}_2=f^{-1}(S_2)$ in $\widehat{S}$ is a sheaf of Azumaya algebras.
\end{prop}

This will be the case for any cubic threefold.
In fact, since we assume from the beginning that a cubic threefold is smooth and the projection line is generic, then $S_1$ is smooth and $S_2$ is empty.

\section{Cubic threefolds}\label{sec:3folds}

This section contains the proofs of our main results on ACM bundles on cubic threefolds.
The goal is to generalize a result of Casanellas--Hartshorne on Ulrich bundles.

As explained in the introduction, the idea is to use Kuznetsov's results on quadric fibrations, to reduce the problem of studying ACM bundles on a cubic threefold to the study of complexes of sheaves on $\PP^2$ with the action of a sheaf of Clifford algebras $\B_0$.

The main technical parts are Sections \ref{subsec:modMd1} and \ref{subsec:modMdgeneral}; there we prove some results on moduli spaces of objects in $\Db(\PP^2,\B_0)$, which are stable with respect to a Bridgeland stability condition.
We come back to Ulrich bundles on cubic threefolds in Section \ref{subsec:Ulrich}.

\subsection{The setting}\label{subsec:3foldsgeom}

Let $Y\subset\PP^4$ be a cubic threefold.
Let $l_0\subseteq Y$ be a general line and consider the blow-up $\uPP$ of $\PP^4$ along $l_0$.
By ``general'' we mean that, if $l$ is any other line meeting $l_0$, then the plane containing them intersects the cubic in three distinct lines (we just avoid the lines of second type, see \cite[Def. 6.6]{CG}).
We set $q: \uPP \to \PP^2$ to be the $\PP^2$-bundle induced by the projection from $l_0$ onto a plane and we denote by $\uY$ the strict transform of $Y$ via this blow-up.
The restriction of $q$ to $\uY$ induces a conic fibration $\pi: \uY \to \PP^2$.
The geometric picture can be summarized by the following diagram
\[
\xymatrix{
D\ar[d]_-{s}\ar@{^{(}->}[r]&\uY \ar@{^{(}->}[rr]^-{\alpha}\ar[d]_-{\sigma}\ar[rrd]_-{\pi}&& \uPP=\PP_{\PP^2}(\O_{\PP^2}^{\oplus 2}\oplus \O_{\PP^2}(-h))\ar[d]_-{q}\\
l_0\ar@{^{(}->}[r]& Y\subset \PP^4 && \PP^2.
}
\]
In particular, the vector bundle $E$ on $S=\PP^2$ introduced in Section \ref{subsec:quadric} is now $\O_{\PP^2}^{\oplus 2}\oplus \O_{\PP^2}(-h)$.
Set $D \subset \uY$ to be the exceptional divisor of the blow-up $\sigma: \uY \to Y$.
We denote by $h$ both the class of a line in $\PP^2$ and its pull-backs to $\uPP$ and $\uY$.
We call $H$ both the class of a hyperplane in $\PP^4$ and its pull-backs to $Y$, $\uPP$, and $\uY$.
We recall that $\O_{\uY}(D)\cong\O_{\uY}(H-h)$, the relative ample line bundle is $\O_{\uPP}(H)$, the relative canonical bundle is $\O_{\uPP}(h-3H)$, and the dualizing line bundle $\omega_{\uY}$ is isomorphic to $\O_{\uY}(-h-H)$ (see, e.g., \cite[Lemma 4.1]{Kuz:4fold}).

\smallskip

The sheaf of even (resp.~ odd) parts of the Clifford algebra corresponding to $\pi$, from Section \ref{subsec:quadric}, specializes in the case of cubic threefolds to
\begin{equation}\label{bcoh}
\begin{split}
\B_0&\cong\O_{\PP^2}\oplus\O_{\PP^2}(-h)\oplus\O_{\PP^2}(-2h)^{\oplus 2}\\
\B_1&\cong\O_{\PP^2}^{\oplus 2}\oplus\O_{\PP^2}(-h)\oplus\O_{\PP^2}(-2h),
\end{split}
\end{equation}
as sheaves of $\O_{\PP^2}$-modules.
The rank $2$ vector bundles $\E'$ and $\E$ sit in the short exact sequences provided by \eqref{eqn:defE'} and \eqref{eqn:defE} respectively, where $L=\O_{\PP^2}(-h)$.

Following \cite{Kuz:4fold} and \cite[Sect.\ 2.1]{BMMS}, one can give a description of the full subcategory $\cat{T}_Y$ in the semiorthogonal decomposition \eqref{eqn:semiorthgen} of $\Db(Y)$.
Indeed, first consider the semiorthogonal decomposition of $\Db(\uY)$ in Theorem \ref{thm:KuzQuadric} and the one
\begin{equation*}
	\Db(\uY)=\langle\sigma^*(\cat{T}_Y),\O_{\uY},\O_{\uY}(H),i_*\O_{D},i_*\O_{D}(H)\rangle
\end{equation*}
obtained by thinking of $\uY$ as the blow-up of $Y$ along $l_0$ and using the main result in \cite{Orlov}.
Then one shows that
\begin{equation*}\label{eqn:last}
	\cat{R}_{\O_{\uY}(-h)}\circ\Phi(\Db(\PP^2,\B_0))=\langle\sigma^*\cat{T}_Y,\O_{\uY}(h-H)\rangle
\end{equation*}
and thus we get a fully faithful embedding
\begin{equation}\label{eqn:Xi3fold}
	 \Xi_3:=(\sigma_*\circ \cat{R}_{\O_{\uY}(-h)}\circ \Phi)^{-1}:\cat{T}_Y\to \Db(\PP^2,\B_0).
\end{equation}
Note that, in view of \cite[Prop.\ 2.9(i)]{BMMS}, $\Xi_3(\cat{T}_Y)$ is the right orthogonal of the category generated by $\B_1$ in $\Db(\PP^2,\B_0)$.

\begin{rem}\label{rmk:Psi&Xi}
	For all $m\in\ZZ$, we have $\Psi(\O_{\uY}(mh))=0$ (see, e.g., \cite[Ex.~2.4]{BMMS}).
	Thus, if $F\in \cat{T}_Y$, then $\Xi_3(F)=\Psi(F)$.
\end{rem}

\subsection{$\B_0$-modules and stability}\label{subsec:modMd1}

Our first goal is to study moduli spaces of stable $\B_0$-modules.
In this section we present how the usual notion of stability extends to our more general situation.

Let $K(\PP^2,\B_0):= K(\Db(\PP^2,\B_0))$ denote the Grothendieck group.
For objects in $\Db(\PP^2,\B_0)$ we can consider the Euler characteristic
\[
\chi(-,-):=\sum_i(-1)^i\hom^i_{\Db(\PP^2,\B_0)}(-,-).
\]
A class $[A]\in K(\PP^2,\B_0)$ in the Grothendieck group is \emph{numerically trivial} if $\chi([M],[A])=0$, for all $[M]\in K(\PP^2,\B_0)$.
We define the \emph{numerical Grothendieck group} $\N(\PP^2,\B_0)$ as the quotient of $K(\PP^2,\B_0)$ by numerically trivial classes.

Given $K\in \Db(\PP^2,\B_0)$, we define its Chern character as
\begin{equation*}
	\ch(K):=\ch(\Forg(K))\in K(\PP^2)\otimes \QQ = H^*(\PP^2,\QQ) \cong \QQ^{\oplus 3},
\end{equation*}
where $\Forg:\Db(\PP^2,\B_0)\to \Db(\PP^2)$ is the functor forgetting the $\B_0$-action.
By linearity the Chern character extends to $K(\PP^2,\B_0)$; it factors through $\N(\PP^2,\B_0)$.

\begin{rem}\label{rmk:numB0}
	(i) By \cite[Prop.~2.12]{BMMS} we have $\N(\PP^2 , \B_0)= \ZZ[\B_1] \oplus \ZZ[\B_0 ] \oplus \ZZ[\B_{-1}]$.
	The Chern characters $\ch(\B_{-1})=(4,-7,\tfrac{15}{2})$, $\ch(\B_0)=(4,-5,\tfrac{9}{2})$, and $\ch(\B_1)=(4,-3,\tfrac{5}{2})$ are linearly independent.
	Hence, the Chern character induces a group homomorphism $\N (\PP^2 ,\B_0)\to K(\PP^2)$ that is an isomorphism over $\QQ$.

	(ii) If $l\subseteq Y$ is a line and $\I_l$ is its ideal sheaf, by \cite[Ex.~2.11]{BMMS}, we have
	\begin{equation}\label{eqn:classXiI}
		\begin{split}
		[\Xi_3(\I_l)]&= [\B_1]-[\B_0] \in \N(\PP^2 , \B_0)\\
		\ch(\Xi_3(\I_l))&= (0,2,-2).
		\end{split}
	\end{equation}

		(iii) Note that $[\B_2]=[\B_{-1}]-3[\B_{0}]+3[\B_{1}]$ and $[\B_{-2}]=3[\B_{-1}]-3[\B_{0}]+[\B_{1}]$.

(iv) Given $[F]=x[\B_{-1}]+y[\B_{0}]+z[\B_{1}]$ or $\ch([F])=(r,c_1,\ch_2)$, we can compute the Euler characteristic as a $\B_0$-module with the following formulas
\begin{equation}\label{eqn:chiB0}
\begin{split}
	\chi(F,F)&=x^2+y^2+z^2+3xy+3yz+6xz\\
	&=-\frac{7}{64}r^2-\frac{1}{4}c_1^2+\frac{1}{2}r\ch_2.
\end{split}
\end{equation}

(v) Let $F\in K(\PP^2,\B_0)$, such that $\ch(F)=(0,2d,z)$.
Then
\begin{equation}\label{eqn:chiB1F}
\chi(\B_1,F)=z+2d
\quad\text{and}\quad \chi(\B_0,F)=z+3d.
\end{equation}

(vi) The Serre functor in $\Db(\PP^2,\B_0)$ is given by $-\otimes_{\B_0}\B_{-1}[2]$ (see, e.g., \cite[Prop.~2.9]{BMMS}).
\end{rem}

We define the Hilbert polynomial of a $\B_0$-module $G$ as the Hilbert polynomial of $\Forg (G)$ with respect to $\O_{\PP^2}(h)$.
Then, the notion of Gieseker (semi)stability is defined in the usual way.
Moduli spaces of semistable $\B_0$-modules have been constructed by Simpson in \cite[Thm.~4.7]{simp}.

We can also consider the slope stability for torsion-free sheaves in $\coh(\PP^2,\B_0)$.
Indeed, we have two natural functions \emph{rank} and \emph{degree} on $\N(\PP^2,\B_0)$:
\begin{align*}
\rk&:\N(\PP^2,\B_0)\to\ZZ,\qquad\rk(K):=\rk(\mathrm{Forg}(K))\\
\deg&:\N(\PP^2,\B_0)\to\ZZ,\qquad\deg(K):=\mathrm{c}_1(\mathrm{Forg}(K)).\mathrm{c}_1(\O_{\PP^2}(h)).
\end{align*}
Given $K\in\coh(\PP^2,\B_0)$ with $\rk(K)\neq0$, we can define the slope $\mu(K):=\deg(K)/\rk(K)$ and the notion of $\mu$-(semi)stability in the usual way.
When we say that $K$ is either torsion-free or torsion of dimension $d$, we always mean that $\mathrm{Forg}(K)$ has this property.

\begin{rem}\label{rmk:Bistable}
As the rank of $\B_0$ and $\B_1$ is $4$, a consequence of \cite[Lemma 2.13(i)]{BMMS} is that these two objects are $\mu$-stable.
Moreover, all morphisms $\B_0\to\B_1$ are injective.
\end{rem}

\begin{lem}\label{lem:Columbus}
Let $A,B\in\coh(\PP^2,\B_0)$ be such that $\ch(A)=\ch(B)$.
Assume one of the following two conditions is satisfied:
\begin{itemize}
\item either $A$ and $B$ are torsion-free sheaves and $\mu$-semistable, or
\item $A$ and $B$ are torsion sheaves pure of dimension $1$ and semistable.
\end{itemize}
Then $\Ext^2(A,B)=0$.
If $A=B$ is actually stable, then $\chi(A,A)\leq1$.
\end{lem}

\begin{proof}
The first claim follows directly from Serre duality.
Indeed, by Remark \ref{rmk:numB0}, (vi), we have
\[
\Ext^2(A,B) = \Hom(B, A\otimes_{\B_0}\B_{-1})^\vee = 0,
\]
since $-\otimes_{\B_0}\B_{-1}$ preserves stability.
For the second, simply observe that
\[
\chi(A,A) = \mathrm{hom}(A,A) - \mathrm{ext}^1(A,A) = 1 - \mathrm{ext}^1(A,A) \leq 1.
\]
\end{proof}

\subsubsection*{Bridgeland stability}

We will need to study stability for objects in $\Db(\PP^2,\B_0)$ which are not necessarily sheaves.
To this end, we briefly recall the concept of Bridgeland stability condition.
For all details we refer to \cite{Br,KS}.

\begin{defn}
A (numerical, full) \emph{Bridgeland stability condition} on $\Db(\PP^2,\B_0)$ consists of a pair $\sigma=(Z,\cat{A})$, where
\begin{itemize}
\item $Z:\N(\PP^2,\B_0)\to\CC$,
\item $\cat{A}$ is the heart of a bounded $t$-structure on $\Db(\PP^2,\B_0)$,
\end{itemize}
satisfying the following compatibilities:
\begin{itemize}
\item[(a)]\label{enum:phase} For all $0\neq G\in\cat{A}$,
\[
Z(G)\in \{ z\in\CC^*:z=|z|\exp(i\pi\phi), \, 0< \phi \leq 1 \};
\]
\item[(b)]\label{enum:HN} Harder--Narasimhan filtrations exist with respect to $\sigma$-stability, namely for any $0\neq G\in\cat{A}$, there is a filtration in $\cat{A}$
\[
0=G_0\subset G_1 \subset \ldots \subset G_N=G
\]
such that $F_i:=G_i/G_{i-1}$ is $\sigma$-semistable and $\phi(F_1)>\ldots>\phi(F_N)$.
\item[(c)] the \emph{support property} holds, namely there exists a constant $C>0$ such that, for all $\sigma$-semistable $F\in \cat{A}$,
\[
\| F \| \leq C \cdot |Z(F)|.
\]
\end{itemize}
\end{defn}

In the previous definition, we used the following notation: by (a), any $0\neq G\in\cat{A}$ has a \emph{phase} $\phi(G):=\frac{1}{\pi}\mathrm{arg}(Z(G))\in(0,1]$.
The notion of $\sigma$-stability in (b) is then given with respect to the phase: $G\in \cat{A}$ is \emph{$\sigma$-(semi)stable} if, for all subobjects $G'\subset G$ in $\cat{A}$, $\phi(G')<(\leq)\phi(G)$.
Finally, we denoted by $\| - \|$ any norm in $\N(\PP^2,\B_0)\otimes\RR$.
The support property is necessary to deform stability conditions and for the existence of a well-behaved wall and chamber structure (this is \cite[Sect.~9]{Br1}; the general statement we need is \cite[Prop.~3.3]{BM}).

%Here we just recall that, by \cite[Prop.\ 5.3]{Br} giving a stability condition on a triangulated category $\cat{D}$ is equivalent to giving a bounded $t$-structure on $\cat{D}$ with heart $\cat{A}$ and a group homomorphism $Z:K(\cat{A})\to\CC$ such that $Z(G)\in\HH$, for all $0\neq G\in\cat{A}$, and with Harder--Narasimhan filtrations.
%Here $\HH:=\{ z\in\CC^*:z=|z|\exp(i\pi\phi), \, 0< \phi \leq 1 \}$.
%More precisely, any $0\neq G\in\cat{A}$ has a \emph{phase} $\phi(G):=\mathrm{arg}(Z(G))\in(0,1]$ and, for $\phi\in(0,1]$, we denote by $\P(\phi)$ the category of $\phi$-semistable objects in $\cat{A}$ of phase $\phi$.
%A stability condition is \emph{locally-finite} if there is some $\eps > 0$ such that, for all $\phi\in\RR$, each (quasi-abelian) subcategory $\P((\phi-\epsilon,\phi+\epsilon))$ (i.e., the category generated by extensions by all semistable objects with phases in the interval $(\phi-\epsilon,\phi+\epsilon)$) is of finite length.
%In this case $\P(\phi)$ has finite length so that every object in $\P(\phi)$ has a finite
%\emph{Jordan--H\"older filtration} into stable factors of the same phase.

%\smallskip

We will only need a special family of stability conditions on $\Db(\PP^2,\B_0)$.

\begin{defn}\label{def:stabfunc}
For $m\in\RR_{>0}$, we define
\begin{equation*}
 \left.\begin{array}{rcl}
Z_{m}:\N(\PP^2,\B_0)&\longrightarrow & \CC \\
{[F]} & \longmapsto&rm^2-\dfrac{9r}{64}-\dfrac{c_1}{2}-\dfrac{\ch_2}{2}+m\sqrt{-1}{\left(r+c_1\right)},
 \end{array}\right.
\end{equation*}
where $\ch([F])=(r,c_1,\ch_2)$.
\end{defn}

By the explicit computations in Remark \ref{rmk:numB0},
\begin{equation}\label{eqn:Z}
\begin{split}
& Z_m([\B_0])=4m^2-\frac{5}{16}-m\sqrt{-1}\\
& Z_m([\B_1])=4m^2-\frac{5}{16}+m\sqrt{-1}\\
&Z_m([\Xi_3(\I_l)])=2m\sqrt{-1}.
\end{split}
\end{equation}

To define an abelian category which is the heart of a bounded $t$-structure on $\Db(\PP^2,\B_0)$, let $\cat{T},\cat{F}\subseteq\coh(S,\beta)$ be the following two full additive
subcategories: The non-trivial objects in $\cat{T}$ are the sheaves $A\in\coh(\PP^2,\B_0)$ such that their torsion-free part has Harder--Narasimhan factors (with respect to $\mu$-stability) of slope $\mu>-1$.
A non-trivial twisted sheaf $A\in\coh(\PP^2,\B_0)$ is an object in $\cat{F}$ if $A$ is torsion-free and every $\mu$-semistable Harder--Narasimhan factor of $A$ has slope $\mu\leq-1$.
It is easy to see that $(\cat{T},\cat{F})$ is a torsion theory and following \cite{Br1}, we define the heart of the induced $t$-structure as the abelian category
\[
\cat{A}:=\left\{A\in\Db(\PP^2,\B_0):\begin{array}{l}
\bullet\;\;\H^i(A)=0\mbox{ for }i\not\in\{-1,0\},\\\bullet\;\;
\H^{-1}(A)\in\cat{F},\\\bullet\;\;\H^0(A)\in\cat{T}\end{array}\right\}.
\]
By Remarks \ref{rmk:numB0} and \ref{rmk:Bistable}, $\B_0[1],\B_1,\Xi_3(\I_l)\in\cat{A}$.

\begin{lem}\label{lem:exstab}
The pair $\sigma_m:=(Z_m,\cat{A})$ defines a stability condition in $\Db(\PP^2,\B_0)$, for all $m>\frac{1}{4}$.
\end{lem}

\begin{proof}
This follows exactly in the same way as in \cite[Prop.~7.1 and Sect.~11]{Br1} and \cite[Prop.~3.13]{Toda:Extremal}.
The only non-standard fact that we need is a Bogomolov--Gieseker inequality for torsion-free $\mu$-stable sheaves.
This is precisely Lemma \ref{lem:Columbus}: for $A\in\coh(\PP^2,\B_0)$ torsion-free and $\mu$-stable, $\chi(A,A)\leq 1$ gives us the desired inequality.

By proceeding as in \cite[Sect.~3]{Toda:Extremal}, to prove the lemma we only have to show property (a) in the definition of stability condition.
Let $A$ be a torsion-free $\mu$-stable sheaf.
Assume further that $\mu(A)=-1$, and so $\mathrm{Im}(Z_m([A]))=0$.
By \eqref{eqn:chiB0} and the fact that $r>0$, we have
\begin{equation*}
 \mathrm{Re}(Z_m([A]))=rm^2-\dfrac{9r}{64}-\dfrac{c_1}{2}-\dfrac{\ch_2}{2}=\frac{1}{r}\left(-\chi(A,A)+m^2r^2-\frac{1}{4}\left(r+c_1\right)^2\right).
\end{equation*}
We need to prove the inequality $\mathrm{Re}(Z_m([A]))>0$, namely $-\chi(A,A)+m^2r^2>0$.
By \cite[Lemma 2.13]{BMMS}, $r\geq 4$, and so for all $m>\frac{1}{4}$, we have $\mathrm{Re}(Z_m([A]))>0$, as we wanted.
\end{proof}

We also observe that all the arguments in \cite{toda} generalize to the non-commutative setting (see also \cite{Lieb:twisted,Lieb}).
In particular, for all $m > \frac{1}{4}$, it makes sense to speak about moduli spaces of $\sigma_m$-semistable objects in $\cat{A}$ as Artin stacks (of finite-type over $\CC$, if we fix the numerical class), and about moduli space of $\sigma_m$-stable objects as algebraic spaces.

\begin{rem}\label{rmk:Bi_sigstab}
As in \cite[Lemma 5.5]{MS}, the objects $\B_0[1]$ and $\B_1$ are $\sigma_m$-stable, for all $m>\frac{1}{4}$.
\end{rem}

\subsection{Moduli spaces of stable $\B_0$-modules: general results}\label{subsec:modMdgeneral}

By keeping in mind \eqref{eqn:classXiI}, we are interested in the following moduli spaces of Gieseker semistable sheaves in $(\PP^2,\B_0)$:

\begin{defn}
Let $d\geq 1$.
We denote by $\MMM_d$ the moduli space of semistable $\B_0$-modules with numerical class $d[\B_1]-d[\B_0]$,
or equivalently, with Chern character $(0,2d,-2d)$.
We denote by $\MMM^s_d\subseteq \MMM_d$ the open subset of stable $\B_0$-modules.
\end{defn}

\begin{ex}\label{ex:objMd}
	Let $C'\subset Y$ be a rational curve of degree $d$.
	Note that, by using for example \cite[Thm.~62]{starr}, one can see that there exists a $2d$-dimensional family of smooth rational curves of degree $d$ on $Y$.
	We can consider the following construction due to Kuznetsov \cite[Lemma 4.6]{Kuz:V14}.
	Set
	\begin{equation}\label{eqn:theta-char}
	 F_d:=   \mathbf{L}_{\O_{Y}}( i_*\O_{C'}(d-1))[-1]   = \ker\left(H^0(Y, i_*\O_{C'}(d-1)) \otimes \O_Y \stackrel{\rm ev}\to i_*\O_{C'}(d-1)\right)\in \cat{T}_Y.
	\end{equation}
	where $i:C'\hookrightarrow Y$.
	Then $\Xi_3(F_d)\in \langle \B_1 \rangle^\perp$.
	Suppose that $C'\cap l_0=\emptyset$.
	Denote by $j$ the composition $C'\hookrightarrow\uY\xrightarrow{\pi}\PP^2$ and suppose that, if we let $C:=j(C')$, the morphism $\res{j}{C'}$ is birational.
	As $C'$ and $l_0$ do not intersect, we can argue exactly as in \cite[Ex.~2.4]{BMMS}. In particular, using that $\Psi(\O_{\widetilde{Y}}(mh))=0$ for all integers $m$, we conclude that
	\begin{equation*}
	\Xi_3(F_d) \cong j_*(\E|_{C'} \otimes \O_{C'}(-1))\otimes\O_{\PP^2}(2h).
	\end{equation*}
	So $\Forg(\Xi_3(F_d))$ is a rank $2$ torsion-free sheaf supported on $C$ and $\Xi_3(F_d)\in \MMM^s_d$.
	The $d=1$ case is treated in Example \ref{ex:RestrLines} below.
	We will also use this example for $d=2$ and $d=3$.
	In such cases, there always exists a curve $C'\subset Y$ with the above properties.
\end{ex}

\begin{ex}\label{ex:RestrLines}
          We can specialize the previous example to the case in which $C'\subset Y$ is a line $l$ which does not intersect $l_0$, namely $d=1$.
          In such a case, we have $F_d\cong I_l$ and
          \begin{equation}\label{eqn:ideal_recta}
	  \Xi_3(\I_{l}) \cong j_*(\E |_{l})\otimes\O_{\PP^2}(h).
	 \end{equation}
	
	 Moreover, we have an isomorphism as $\O_{\PP^2}$-modules
	 \[
	  \Xi_3(\I_{l}) \cong \O_l\oplus\O_l(-h).
	 \]
	 Indeed, by \eqref{eqn:classXiI}, the Chern character of $\Xi_3(\I_l)$ as $\O_{\PP^2}$-module is
	 \[
	  \ch(\Xi_3(\I_l))=\left(0,2,-2\right),
	 \]
	 Therefore, $\Xi_3(\I_l)\cong\O_l(a)\oplus\O_l(-1-a)$, for $a\in\ZZ_{\geq0}$.
	 Since, by \cite[Lemma 4.8]{BMMS}, we have
	 \[
	  0 = \Hom_{\Db(\PP^2,\B_0)}(\B_0(h),\Xi_3(\I_l)) = \Hom_{\PP^2}(\O_{\PP^2}, \O_l(a-1)\oplus\O_l(-2-a)),
	 \]
	 we deduce that $a=0$, as we wanted.
\end{ex}

It is a standard fact (it follows, e.g., as in \cite[Example 9.5]{BM:projectivity}) that the assignment
\begin{equation*}
	\varUpsilon: \MMM_d \to \abs{\O_{\PP^2}(d)} \qquad \qquad G\mapsto \supp\Forg (G).
\end{equation*}
extends to a morphism which is well-defined everywhere.
Theorem \ref{thm:mainB0} becomes then the following statement:

\begin{thm}\label{thm:Md}
	The moduli space $\MMM_d$ is irreducible and, for a general smooth curve $C\in \abs{\O_{\PP^2}(d)}$, we have
	\begin{equation*}
		\varUpsilon^{-1}(C)\cong \bigsqcup_{2^{5d-1}} JC,
	\end{equation*}
where $JC=\set{L\in \Pic(C) \st L \text{ is algebraically equivalent to }\O_C}$ is the Jacobian of $C$. Moreover, the stable locus $\MMM^s_d$ is smooth of dimension $d^2+1$.
\end{thm}

Before proceeding with the general proof which is carried out in the next section, we examine the easy case $d=1$:

\begin{prop}\label{prop:d=1}
	The moduli space $\MMM_1=\MMM^s_1$ is isomorphic to the Fano surface of lines $F(Y)$ blown-up at the line $l_0$.
	In particular, $\MMM_1$ is smooth and irreducible.
\end{prop}

To prove Proposition \ref{prop:d=1}, we use wall-crossing techniques from \cite{BM:projectivity} for the family of Bridgeland stability conditions $\sigma_m$ of Lemma \ref{lem:exstab}.
The precise result we need is the following lemma, whose proof is exactly the same as \cite[Lemma 5.7]{MS}.

\begin{lem} \label{lem:stabM1}
Let $F\in \MMM_1$.

{\rm (i)} If $F\in\gen{\B_1}^\perp$ then $F$ is $\sigma_m$-stable for all $m>\frac{1}{4}$.
Moreover, in this case $F\in\cat{A}\cap\Xi_3(\cat{T}_Y)$.
By \cite[Thm.~4.1]{BMMS}, $F\cong\Xi_3(\I_l)$, for some line $l\neq l_0$ in $Y$.

{\rm (ii)} If $F\not\in\gen{\B_1}^\perp$, then $F$ sits in a short exact sequence
\begin{equation}\label{eqn:extdest}
	0\to\B_0\to\B_1\to F\to 0,
\end{equation}
and $F$ becomes $\sigma_m$-semistable for $m=\frac{\sqrt{5}}{8}$ with Jordan--H\"older filtration
\begin{equation*}
\B_1\to F\to \B_0[1].
\end{equation*}
\end{lem}

By \cite[Example 2.11]{BMMS}, the object $\Xi_3(\I_{l_0})$ sits in the distinguished triangle
\begin{equation*}
\B_0[1]\to \Xi_3(\I_{l_0})\to\B_1,
\end{equation*}
which is the Harder--Narasimhan filtration of $\Xi_3(\I_{l_0})$ for $m > m_0:=\frac{\sqrt{5}}{8}$.
Thus, all such extensions \eqref{eqn:extdest} get contracted to $\Xi_3(\I_{l_0})$ which is indeed $\sigma_m$-stable for $m\in (m_0-\eps , m_0 )$.
The wall-crossing phenomenon described in \cite[Sect.\ 5.2]{MS} carries over and this proves Proposition \ref{prop:d=1}.

\subsection{Proof of Theorem \ref{thm:Md}}\label{subsec:proof}
	The argument is divided into various steps.

	\bigskip

\noindent {\bf Step 1:} {\em Deformation theory.}
	For any $G\in \MMM_d$, we have $\chi(G,G)=-d^2$.
	Hence, to prove that $\MMM^s_d$ is smooth of dimension $d^2+1$ it is enough to show that, it is non-empty and
	\begin{equation*}
	 \HomDB^2(G,G)=0,
	\end{equation*}
	for any $G\in \MMM_d$. The fact that $\MMM^s_d$ is non-empty is a consequence of the next step.
	The vanishing of $\HomDB^2(G,G)$ follows directly from Lemma \ref{lem:Columbus}.

	\bigskip

\noindent {\bf Step 2:} {\em Fibres of $\varUpsilon:\MMM_d\to \abs{\O_{\PP^2}(d)}$.}
	We claim that for a smooth curve $C\in \abs{\O_{\PP^2}(d)}$,
	\begin{equation}\label{eqn:fibr}
	\varUpsilon^{-1}(C)\cong \bigsqcup_{2^{5d-1}} \Pic^0(C).
	\end{equation}

	Recall that the conic fibration $\pi$ degenerates along a smooth quintic $\Delta\subset\PP^2$. We denote $\res{\Delta}{C}$ by $\sum_{i=1}^{5d} p_i$ (the points are possibly non-distinct) and,
	abusing notation, we set $\tfrac{1}{2}p_i$ to be the the section in $\widehat C$ corresponding to the $2^{\rm nd}$-root of $p_i$.
	As in Proposition \ref{prop:AzumayaStack}, we can consider the stack $\widehat{\PP^2}$ over $\PP^2$ of $2^{\rm nd}$ roots of $\O_{\PP^2}(\Delta)$ along the section $\Delta$.
	We denote by $\psi: \widehat{\PP^2} \to \PP^2$ the natural projection.
	We then have an equivalence of abelian categories
	\begin{equation*}\label{eqn:hatPP2}
	 \psi_*:\cat{Coh}(\widehat{\PP^2}, \A_0)\to \cat{Coh}(\PP^2,\B_0).
	\end{equation*}

	Given a smooth curve $C\subset \PP^2$ we can restrict this construction to $\psi:\widehat C\to C$ where $\widehat C$ is a twisted curve (stack of $2^{\rm nd}$ roots of $(C,\res{\Delta}{C})$).
	The restriction $\res{\A_0}{\widehat C}$ is a sheaf of (trivial) Azumaya algebras, i.e., there exists a vector bundle of rank 2, $E_{C,0}\in \cat{Coh}(\widehat{C})$, such that $\res{\A_0}{\widehat C}=\mathcal{E}nd(E_{C,0})$ (see, for example, \cite[Cor.~3.16]{Kuz:Quadric}) and
	\begin{equation*}\label{eqn:equivalence}
	 \left.\begin{array}{rcccl}
	 {\cat{Coh}(\widehat C)}&\stackrel{\sim}\lra &{\cat{Coh}(\widehat C,\res{\A_0}{\widehat C})} &\stackrel{\sim}\lra &{\cat{Coh}(C,\res{\B_0}{C})}\\
	 {G}&\longmapsto& {G\otimes E_{C,0}^\vee}&\longmapsto& {\psi_{*}(G\otimes E_{C,0}^\vee)}
	 \end{array}\right.
	\end{equation*}
	is an equivalence of categories. In particular,
	\begin{equation*}\label{eqn:e}
	\res{\B_0}{C}=\psi_*(\cEnd(E_{C,0})).
	\end{equation*}
	Moreover, there certainly exists $M\in\Pic(\widehat C)$ such that $\ch_2(\psi_*(E_{C,0}^\vee\otimes M))=-2d$ as an $\O_{\PP^2}$-module.
	Thus, $\psi_*(E_{C,0}^\vee\otimes M)\in{\MMM}_d$.
	Since $E_{C,0}$ is determined up to tensorization by line bundles, we can assume directly that $\psi_*E_{C,0}^\vee\in {\MMM}_d$.

	As $E_{C,0}$ is a rank two vector bundle on $\widehat C$, it is clear that the fiber of $\varUpsilon$ over the smooth curve $C$ consists of line bundles on $\widehat{C}$.
	By \cite[Cor.~3.1.2]{Ca}, an invertible sheaf on $\widehat C$ is of the form $\psi^*L\otimes \O_{\widehat C}\Big(\sum_{i=1}^{5d} \tfrac{\lambda_i}{2} p_i\Big)$, where $L\in \Pic(C)$ and $\lambda_i\in\set{0,1}$.
	On the other hand, as $E_{C,0}^\vee$ has rank $2$, we have
	\begin{equation*}
		\ch_2\Big(\psi_*\Big(E_{C,0}^\vee\otimes \psi^*L\otimes \O_{\widehat C}\Big(\sum_{i=1}^{5d} \tfrac{\lambda_i}{2} p_{i}\Big)\Big)\Big)=\ch_2\Big(\psi_*(E_{C,0}^\vee)\Big),
	\end{equation*}
	as objects in $\Db(\PP^2)$ with $L\in\Pic(C)$ and $\lambda_i\in\{0,1\}$, if and only if $2 \deg L+\sum_{i=1}^{5d}\lambda_i =0$.

	Let $J$ be the set of all subsets of $\set{1,\ldots,5d}$ of even cardinality and, for $I\in J$, set $\tau_I$ to be the cardinality of $I$.
	Then the above discussion can be rewritten as
	\begin{equation*}
		\varUpsilon^{-1}(C)=\bigsqcup_{I\in J}\left\{\psi_*\Big(E_{C,0}^\vee\otimes \psi^*L\otimes \O_{\widehat C}\Big(\sum_{i\in I} \tfrac{1}{2} p_{i}\Big)\Big):L\in\Pic^{-\tau_I}(C)\right\}.
	\end{equation*}
	Hence
	\begin{equation*}
		\varUpsilon^{-1}(C)\cong\bigsqcup_{I\in J}\Pic^{-\tau_I}(C),
	\end{equation*}
	which is precisely \eqref{eqn:fibr}, because $J$ has cardinality $2^{5d-1}$.

	\bigskip

\noindent {\bf Step 3:} {\em $\MMM_d$ is irreducible.}
	To prove the irreducibility of $\MMM_d$, we follow the same strategy as in \cite{KLS}.
	We first prove that $\MMM_d$ is connected by simply following the same argument as in the proof of \cite[Thm.\ 4.4]{KLS}.
	Indeed, by Proposition \ref{prop:d=1}, we know that $\MMM_1$ is connected.
	Now if $1\leq d_1\leq d_2<d$ and $d_1+d_2=d$, we have the natural maps $\varphi_{d_1,d_2}:\MMM_{d_1}\times \MMM_{d_2}\to \MMM_d$ sending the pair $(E_1,E_2)$ to $E_1\oplus E_2$.
	Their images coincide with the semistable locus of $\MMM_d$, which is then connected by induction.
	The existence of a connected component in $\MMM_d$ consisting of purely stable objects can be excluded by using an argument of Mukai as in the proof of \cite[Thm.\ 4.1]{KLS}. For the convenience of the reader, let us outline here some details.

	The aim is to show that, if there exists a connected component $X$ of the moduli space $\MMM_d$ consisting only of stable sheaves, then $X\cong \MMM_d$. Take a point $F\in X$ and assume there is $G\in\MMM_d\setminus X$.
	The arguments in \cite[Lemma 4.2]{KLS} and in \cite[\S 4.3]{KLS} show that we can essentially assume that there is a universal family $\F\in\coh(X\times\MMM_d)$ with two projections $p:X\times\MMM_d\to X$ and $q:X\times\MMM_d\to\MMM_d$.
	By Lemma \ref{lem:Columbus}, we have that $\HomDB^2(G',F')=0$, for all $G',F'\in\MMM_d$.
	Hence a computation of local Exts shows that
	\[
	\Ext^0_p(q^*G,\F)=\Ext^2_p(q^*G,\F)=0,
	\]
	while $\Ext_p^1(q^*G,\F)$ is locally free of rank $d^2<1+d^2=\dim(X)$ (the last equality follows again from Step 1).
	If we replace $G$ by $F$, we get a complex of $\O_X$-modules
	\begin{equation}\label{eqn:Mukai1}
		\xymatrix{
	0\ar[r]& A_0\ar[r]^-{\alpha}& A_1\ar[r]^-{\beta}&A_2\ar[r]& 0
	}
	\end{equation}
	such that $H^i(A^\bullet)\cong\Ext^i_p(q^*F,\F)$ (more generally, this holds for any base change $S\to X$).

	It turns out that the point $F\in X$ is the degeneracy locus of the map $\alpha$ (see \cite[Lemma 4.3]{KLS}).
	Thus, blowing up $X$ at $F$, we get $f:Z\to X$ providing, as in \eqref{eqn:Mukai1}, a new complex of $\O_Z$-modules
	\begin{equation*}\label{eqn:Mukai2}
		\xymatrix{
	A'_0\ar[r]^-{\alpha'}& A'_1\ar[r]^-{\beta'}&A'_2
	}
	\end{equation*}
	with an inclusion $f^*A_0\subseteq A'_0$.
	Let $D$ be the exceptional divisor in $Z$ and let $W'$ be the middle cohomology of
	\begin{equation*}\label{eqn:Mukai3}
		\xymatrix{
	0\ar[r]&A'_0\ar[r]& f^*A_1\ar[r]&A'_2\ar[r]&0
	}
	\end{equation*}
	If $M:=A_0'/f^*A_0$, then, by the same computations of the Chern classes as in \cite[\S 4.4]{KLS}, we get $c(W')=f^*c(W)\cdot c(-M)$.
	As the rank of $W$ and $W'$ are smaller than $\dim(X)$, one gets a contradiction as $c_{\dim(X)}(W')=0$ while, using that $M\cong\O_D(D)$, one shows that the component in degree $\dim(X)$ of $f^*c(W)\cdot c(-M)$ is not trivial.

	Hence, to conclude that $\MMM_d$ is irreducible, it is enough to show that it is normal.
	Since, by Lemma \ref{lem:Columbus}, $\HomDB^2(G,G)=0$ for all $G\in \MMM_d$, the Kuranishi map is trivial.
	Following then the notation of \cite[\S 2.7]{KLS}, the quadratic part $\mu$ of the Kuranishi map is also trivial and the null-fibre $F=\mu^{-1}(0)$ coincides with $\HomDB^1(G,G)$, which is obviously normal.
	Then we can apply \cite[Prop.~3.8]{KLS}, i.e., if we consider a slice $S$ of an orbit of a semistable point $[q]$ of the corresponding $\mathrm{Quot}$-space, we have that $\O_{S,[q]}$ is a normal domain.
	Since being normal is an open property, we can use the arguments in the proof of \cite[Prop.~3.11]{KLS} to prove that the locus $R^{ss}$ of semistable points of the $\mathrm{Quot}$-space is normal.
	A GIT-quotient of a normal scheme is normal.
	Hence, $\MMM_d$ is normal, since it is a GIT-quotient of $R^{ss}$.
	
\bigskip

This finally concludes the proof of Theorem \ref{thm:Md}.

\begin{rem}
	When $d=1$, the map $\varUpsilon$ has a very natural and well-known geometrical interpretation.
	In fact, given a $\B_0$-module $F$ supported on a general line $l\subset \PP^2$, we can consider all the lines $l'$ in $Y$ such that $\Xi_3(\I_{l'})\cong F$.
	By Proposition \ref{prop:d=1}, we have to count the number of lines $l'$ that map to $l$ via the projection from $l_0$ (where the lines that intersect $l_0$ are mapped to the projection of the tangent space of the intersection point).
	The lines that intersect $l_0$ form an Abel--Prym curve in $F(Y)$, so they do not dominate $\abs{\O_{\PP^2}(1)}$.
	Hence, we need only to count the skew lines to $l_0$ that map to $l$.
	The preimage of $l$ via the projection is a cubic surface, so it contains 27 lines.
	The line $l$ intersects the degeneration quintic $\Delta$ in $5$ points, which give us $5$ coplanar pairs of lines intersecting $l_0$.
	Hence we have $27-10-1=2^4$ lines skew from $l_0$ that project to $l$.
	Indeed, if $\Bl_{l_0}F(Y)$ is the blow-up of $F(Y)$ along $l_0$, we have a finite morphism $\Bl_{l_0}F(Y)\to \abs{\O_{\PP^2}(1)}$ which is $2^4:1$ (see, e.g., \cite[Proof of Thm.~4]{be2}).
\end{rem}

For applications to stable sheaves on cubic threefolds, as in Lemma \ref{lem:stabM1}, we consider the subset
\begin{equation*}\label{eqn:defNd}
	\NNN_d:=\gen{\B_1}^{\perp}\cap \MMM_d.
\end{equation*}

\begin{lem}\label{lem:NdWellDefined}
The subset $\NNN_d$ is well-defined, namely it does not depend on the chosen representative in the S-equivalence class.
\end{lem}

\begin{proof}
First of all we observe that by Remark \ref{rmk:numB0}, (v), for all $A\in\MMM_d$, we have $\chi(\B_1,A)=0$.
Moreover, by Serre duality, $\Hom^2(\B_1,A)=0$, since $A$ is torsion.
Hence, we have $\Hom(\B_1,A)=0$ if and only if $A\in\NNN_d$.

Let $A\in\MMM_d$ and let $A_1,\ldots,A_m$ be its Jordan-H\"older factors.
It is enough to show the following claim: $A\in\NNN_d$ if and only if $A_i\in\NNN_d$, for all $i$.
As observed in Step 3 of the proof of Theorem \ref{thm:Md}, we have $A_i\in\MMM_{d_i}$, for some $d_i>0$.
The claim follows then directly from the long exact sequence in cohomology (by applying $\Hom(\B_1,-)$ to the Jordan-H\"older filtration of $A$) and by the previous observation.
\end{proof}

By semi-continuity, the condition of belonging to $\NNN_d$ is open in families.
We also have the following result:

\begin{prop}\label{prop:Nd}
	The subset $\NNN_d$ is non-empty and dense in $\MMM_d$.
\end{prop}

\begin{proof}
This is a well-known general fact.
The proof we give here mimics \cite[Theorem 2.15]{BM:MMP}.
We first recall that, as proved in Step 3 of the proof of Theorem \ref{thm:Md}, by considering the maps $\phi_{d_1,d_2}:\MMM_{d_1}\times\MMM_{d_2}\to\MMM_d$, the subset $\MMM_d^s$ consisting of stable sheaves in $\MMM_d$ is open, non-empty, and dense.

We can now proceed by induction on $d$.
The case $d=1$ is precisely Lemma \ref{lem:stabM1}.

Assume then $d>1$, and let $A_{d-1}\in\NNN_{d-1}$.
Since $\MMM_d^{s}$ is non-empty, we can assume $A_{d-1}$ is stable.
Then we can find $A_1\in\NNN_1$ such that $\Hom(A_{d-1},A_1)=\Hom(A_1,A_{d-1})=0$ (in fact, any $A_1\in\NNN_1$ works, since $A_{d-1}$ is stable).
By Remark \ref{rmk:numB0}, (v), we have
\[
\chi(A_1,A_{d-1}) = (d-1) \chi(A_1,A_1) = - (d-1) < 0.
\]
Hence, $\Ext^1(A_1,A_{d-1})\neq0$.
Consider a non-trivial extension
\[
0 \to A_{d-1} \to A_d \to A_1 \to 0.
\]
Then $A_d\in\NNN_d$ and $\Hom(A_d,A_d)\cong\CC$, namely $A_d$ is a simple sheaf.
Since, by Lemma \ref{lem:Columbus}, $\Ext^2(A_d,A_d)=0$, we can consider a maximal dimensional family of simple sheaves containing $A_d$.
Hence, since both being stable and belonging to $\NNN_d$ is an open property, we have
\[
\emptyset\neq\NNN_d\cap\MMM_d^s\subset\MMM_d^s
\]
is an open subset, and therefore dense.
Since $\MMM_d^s$ is dense in $\MMM_d$, this concludes the proof.
\end{proof}

\subsection{Ulrich bundles}\label{subsec:Ulrich}

We now apply the results on $\B_0$-modules of the previous section to study Ulrich bundles on a cubic threefold $Y$.
The goal is to prove Theorem \ref{thm:main3folds} from the introduction.

\begin{defn}
	An ACM bundle $F$ on $Y$ is called \emph{Ulrich} if the graded module $H^0_*(Y,F):=\bigoplus_{m\in\ZZ}H^0(Y,F(mH))$ has $3\rk(F)$ generators in degree $1$.
\end{defn}

We refer to \cite[Sect.~1]{CH} for the basic properties of Ulrich bundles on projective varieties.
In particular, we recall the following presentation of stable Ulrich bundles due to the Hartshorne--Serre construction.

\begin{lem}\label{lem:Ulr_geom}
	A stable Ulrich bundle $F$ of rank $r$ on a cubic threefold $Y$ admits the following presentation
	\begin{equation}
		0\to \O_Y(-H)^{\oplus r-1}\to F\to \I_C\otimes\O_Y((r-1)H)\to 0,
	\end{equation}
	where $C$ is a smooth connected curve of degree $\frac{3r^2-r}{2}$ and arithmetic genus $r^3-2r^2+1$.
\end{lem}
\begin{proof}
	By definition, $F(H)$ is generated by global sections, so $G:=\coker (\O_Y(-H)^{\oplus r-1}\hookrightarrow F)$ is a torsion-free sheaf of rank $1$.
	By choosing the sections appropiately, we get that $G=\I_C\otimes \O_Y(sH)$, where $C\subset Y$ is a smooth curve.
	By \cite[Lemma 2.4(iii)]{CH}, we have $c_1(F(H))=r$, so $s=r-1$.
	Since $h^1(Y,\I_C)=0$, $C$ is connected.
	By \cite[Prop. 3.7]{CH1}, we have that $\deg C=\frac{3r^2-r}{2}$ and by Riemann--Roch we get $p_a(C)=r^3-2r^2+1$.
\end{proof}

From Lemma \ref{lem:Ulr_geom} it is standard to compute the Chern character of an Ulrich bundle $F$ of rank $r$, by using Hirzebruch--Riemann--Roch:
\begin{equation}\label{eqn:chUlr}
	\ch(F)=(r,0,-r\cdot l,0),
\end{equation}
where $l$ denotes the class of a line in $Y$.

\begin{rem}\label{rmk:norm}
	Notice that, in their definition of Ulrich bundles \cite[Def. 2.1]{CH}, Casanellas and Hartshorne impose the generators to be in degree $0$.
	Hence, their Ulrich bundles can be obtained from ours by twisting by $\O_Y(H)$ and viceversa.
	We prefer this normalization, since then Ulrich bundles are \emph{balanced} ACM bundles (recall Definition \ref{def:ACMbalanced}).
	Moreover, with this normalization, instanton bundles of minimal charge (see the forthcoming Definition \ref{def:instantons} and the subsequent comments) are also Ulrich bundles.
\end{rem}

Denote by $\MMM^{sU}_r$ the moduli space of stable Ulrich bundles of rank $r\geq 2$.
It is smooth of dimension $r^2+1$ since for any such bundle $E$, we have $\dim\Ext^1(E,E)=r^2+1$ while $\dim\Ext^2(E,E)=0$.

To prove that $\MMM^{sU}_r$ is non-empty, the strategy is to show the existence of low rank Ulrich bundles ($r=2,3$) and then use a ``standard'' deformation argument \cite[Thm. 5.7]{CH}.
The existence of rank $2$ Ulrich bundles is well-known \cite{Druel,MT}.
They usually appear in the literature as instanton bundles (see the forthcoming Section \ref{sec:modMd2}).
In \cite{CH} the authors construct rank $3$ Ulrich bundles, relying on the existence of an ACM curve on $Y$ of degree $12$ and genus $10$ (see Lemma \ref{lem:Ulr_geom}).
The existence of such curves is proved, using {\tt Macaulay2}, by Gei{\ss} and Schreyer in the appendix and only for a generic cubic threefold.

Our approach to construct Ulrich bundles of rank $3$ is different (for completeness we also construct rank $2$ Ulrich bundles).
In particular, we do not use the Hartshorne--Serre construction (see Lemma \ref{lem:Ulr_geom}), but the structure of conic fibration of a blow-up of $Y$.
We have computed the image in $\Db(\PP^2,\B_0)$ of the ideal sheaves of lines in $Y$ in Example \ref{ex:RestrLines}.
We can therefore consider extensions of them, and use deformation theory to cover the subset $\NNN_d\subset \MMM_d$ (for $d=2,3$).
If $G$ is a general sheaf in $\NNN_d$, then the object $\Xi_3^{-1}(G)$ will be a stable ACM bundle of rank $d$, which will be automatically Ulrich.

The advantage of our approach is that by using the category $\cat{T}_Y$ we are able to reduce all computations to the category $\Db(\PP^2,\B_0)$, via the functor $\Xi_3$.
Thus, the existence result needed goes back to Theorem \ref{thm:Md}.

\medskip

Given $G\in \NNN_d$, we want to study $\Xi_3^{-1}(G)\in \cat{T}_Y$.
In order to show that it is an ACM bundle we want to see how the vanishings in Lemma \ref{lem:viceversa} can be checked in $\Db(\PP^2,\B_0)$.

\begin{lem}\label{lem:TgBundle}
	We have the following natural isomorphisms
\begin{equation*}
	\begin{split}
	\RHom_{\Db(Y)}(\O_Y(2H),F)[2]&\cong\RHom_{\Db(\PP^2,\B_0)}(\Omega_{\PP^2}(2h)\otimes\B_0,\Xi_3(F))\\&\cong\RHom_{\Db(\PP^2)}(\Omega_{\PP^2}(2h),\Xi_3(F)),\\
	\RHom_{\Db(Y)}(\O_Y(-H),F)&\cong\RHom_{\Db(\PP^2,\B_0)}(\B_{-1},\Xi_3(F))\\&\cong\RHom_{\Db(\PP^2)}(\O_{\PP^2},\mathrm{Forg}(\Xi_3(F\otimes_{\B_0}\B_1))),
	\end{split}
\end{equation*}
	for all $F\in \cat{T}_Y$.
\end{lem}

\begin{proof}
	As for the first series of isomorphisms, we start with the following chain of natural isomorphisms, which follows directly from the definitions:
	\begin{equation}\label{eqn:RHom2H-F}
		\begin{split}
		\RHom_{\Db(Y)}(\O_Y(2H),F) &\cong\RHom_{\Db(Y)}(\O_Y(2H),\Xi_3^{-1}(\Xi_3(F)))\\
		&\cong\RHom_{\Db(Y)}(\O_Y(2H),\sigma_*\circ\cat{R}_{\O_{\uY}(-h)}\circ\Phi\circ\Xi_3(F))\\
		&\cong\RHom_{\Db(\PP^2,\B_0)}(\Psi\circ\cat{L}_{\O_{\uY}(H)}(\O_{\uY}(2H)),\Xi_3(F)).
		\end{split}
	\end{equation}
	By \eqref{eqn:LRmutation}, $\cat{L}_{\O_{\uY}(H)}\O_{\uY}(2H)$ is given by

	\begin{equation}\label{eqn:Lmut2H}
		\cat{L}_{\O_{\uY}(H)}\O_{\uY}(2H)= \mathrm{cone}\left(\O_{\uY}(H)^{\oplus 5}\xrightarrow{\mathrm{ev}}\O_{\uY}(2H)\right).
	\end{equation}

	By definition of $\Psi$ \eqref{eqn:bascoh}, we have the two exact triangles in $\Db(\PP^2,\B_0)$:
	\begin{align*}
		\B_1[1] \to \B_0(h)\otimes q_*\O_{\widetilde \PP^4}(H)[1]\to &\Psi(\O_{\uY}(2H))\\
		\B_1\otimes q_*\O_{\widetilde \PP^4}(-H)[1] \to \B_0(h)[1]\to &\Psi(\O_{\uY}(H)).
	\end{align*}

	Since $\Xi_3(F)\in \langle\B_1\rangle^{\perp}$ and $q_*\O_{\widetilde \PP^4}(H)=\O_{\PP^2}^{\oplus 2}\oplus \O_{\PP^2}(h)$, we have
	\begin{align*}
		\RHom_{\Db(\PP^2,\B_0)}&(\Psi(\O_{\uY}(2H)),\Xi_3(F))=\RHom_{\Db(\PP^2,\B_0)}(\B_0(h)\otimes q_*\O_{\widetilde \PP^4}(H)[1],\Xi_3(F)) \\
		&=\RHom_{\Db(\PP^2,\B_0)}(\B_0\otimes (\O_{\PP^2}(h)^{\oplus 2}\oplus \O_{\PP^2}(2h))[1],\Xi_3(F)).
	\end{align*}
	and since $q_*\O_{\widetilde \PP^4}(-H)=0$,
	\begin{align*}
		\RHom_{\Db(\PP^2,\B_0)}(\Psi\O_{\uY}(H),\Xi_3 F)&=\RHom_{\Db(\PP^2,\B_0)}(\B_0(h)[1],\Xi_3(F)).
	\end{align*}

	Therefore, combining \eqref{eqn:RHom2H-F} and \eqref{eqn:Lmut2H} we have
	\begin{align*}
		\RHom_{\Db(Y)}&(\O_Y(2H),F)\cong \\
		&\cong\RHom_{\Db(\PP^2,\B_0)}\big(\B_0\otimes\mathrm{cone}(\O_{\PP^2}(h)^{\oplus 5} \xrightarrow{\mathrm{ev}} \O_{\PP^2}(h)^{\oplus 2}\oplus \O_{\PP^2}(2h))[1],\Xi_3(F)\big)\\
		&\cong\RHom_{\Db(\PP^2,\B_0)}\big(\B_0\otimes \Omega_{\PP^2}(2h),\Xi_3(F)\big)[-2]\\
		&\cong\RHom_{\Db(\PP^2)}\big(\Omega_{\PP^2}(2h),\Xi_3(F)\big)[-2],
	\end{align*}
	where we have used that $\mathrm{cone}(\O_{\PP^2}(h)^{\oplus 5} \xrightarrow{\mathrm{ev}} \O_{\PP^2}(h)^{\oplus 2}\oplus \O_{\PP^2}(2h))\cong\Omega_{\PP^2}(2h)[1]$, and the first row of isomorphisms follows.

	It remains to prove the isomorphisms in the second line of the lemma.
	We start with the following chain of natural isomorphisms, which follows directly from the definitions:
	\begin{equation}\label{eqn:RHom-H-F}
		\begin{split}
		\RHom_{\Db(Y)}(\O_Y(-H),F) &\cong\RHom_{\Db(Y)}(\O_Y(-H),\Xi_3^{-1}(\Xi_3(F)))\\
		&\cong\RHom_{\Db(Y)}(\O_Y(-H),\sigma_*\circ\cat{R}_{\O_{\uY}(-h)}\circ\Phi\circ\Xi_3(F))\\
		&\cong\RHom_{\Db(\PP^2,\B_0)}(\Psi\circ\cat{L}_{\O_{\uY}(H)}(\O_{\uY}(-H)),\Xi_3(F))\\
		&\cong\RHom_{\Db(\PP^2,\B_0)}(\Psi(\O_{\uY}(-H)),\Xi_3(F))\\
		&\cong\RHom_{\Db(\PP^2,\B_0)}(\B_{-1},\Xi_3(F)).
		\end{split}
	\end{equation}
	The last isomorphism is an easy computation, and the lemma follows.
\end{proof}

Now we are ready to give a geometric interpretation of the objects of $\NNN_d$ (recall \eqref{eqn:defNd}).

\begin{prop}\label{prop:exACM}
	If $d=2,3$ and $G$ is a general sheaf in $\NNN_d$, then the object $\Xi_3^{-1}(G)$ is a stable ACM bundle of rank $d$.
\end{prop}

\begin{proof}
	Again the argument can be divided in a few parts.

	\bigskip

	\noindent {\bf Step 1:} {\em $\Xi_3^{-1}G$ is a coherent sheaf.} By Example \ref{ex:objMd}, the sheaf $F_d$ is in $\cat{T}_Y$ and $\Xi_3(F_d)\in \NNN_d$.
	By semi-continuity, for $G\in \NNN_d$ general, the object $\Xi_3^{-1}(G)$ has to be a sheaf.

	\bigskip

	\noindent {\bf Step 2:} {\em First vanishing.}
	We want to show that
	$H^i(Y,\Xi_3^{-1}(G)\otimes \O_Y(-2H))=0$ for $i=1,2$.
	By Lemma \ref{lem:TgBundle}, we need to prove that $\Hom^0_{\Db(\PP^2)}(\Omega_{\PP^2}(2h),G)=0$.

	Before that, by Example \ref{ex:RestrLines}, we observe that
	\begin{equation*}
	\Hom^0_{\Db(\PP^2,\B_0)}(\B_0,\Xi_3(\I_l)) = \Hom^0_{\PP^2}(\O_{\PP^2}, \O_{l}\oplus\O_{l}(-1))\cong\CC,
	\end{equation*}
	and all other Hom-groups are trivial.
	The extension of $d$ sheaves $\Xi_3(\I_l)$ (with different $l$) lies in $\MMM_d^s$.
	By semicontinuity and induction on $d$,
	\begin{equation}\label{eqn:dHoms}
	\Hom^0_{\Db(\PP^2,\B_0)}(\B_0,G)\cong H^0(\PP^2,G) \cong\CC^d,
	\end{equation}
	for $G$ general $\MMM_d$. Notice that here we are implicitly using that $\chi(\B_0,G)=d$. Indeed, this follows from \eqref{eqn:chiB1F}.

	\smallskip

	\noindent {\bf Case 1:} {\em Rank $d=2$.}
	Let $G$ be supported in $j_C:\PP^1\hookrightarrow C\in \abs{\O_{\PP^2}(2)}$.
	Let $C$ be smooth and intersecting $\Delta$ transversally.
	Note that $\res{\Omega_{\PP^2}(2h)}{C}\cong i_*(\O_{\PP^1}(1)^{\oplus 2})$, so we have to show that $H^0(\PP^1, G(-1))=0$.
	By \eqref{eqn:dHoms} and semi-continuity, $G$ has only two possibilities (as $\O_{\PP^2}$-module):
	\begin{align}
	&G \cong j_{C_*}\left(\O_{\PP^1}(1)\oplus\O_{\PP^1}(-1) \right)\label{eqn:dolent},\\
	&G \cong \O_C^{\oplus 2}\label{eqn:bo}.
	\end{align}

	If we are in situation \eqref{eqn:bo}, then the desired vanishing holds.

	Assume now that $G \cong j_{C_*}(\O_{\PP^1}(1)\oplus\O_{\PP^1}(-1))$.
	Recall that $G=\psi_*(L\otimes E_{C,0}^\vee)$ for some $L\in \Pic(\widehat C)$.
	We use a method from \cite{Beauvillecubic, Iliev}.
	The projective bundle $\PP(L\otimes E_{C,0}^\vee)\to \widehat C$ corresponds by definition to the conic bundle over $\widehat C$ induced by the conic fibration $\pi: \uY \to \PP^2$.
	More precisely, $\pi^{-1}(C)$ is a conic bundle over $C$ with 10 singular fibres $\pi^{-1}(C\cap \Delta) = \bigcup_{i=1}^{10} l_i \cup l_i'$.
	The lines $l_i$ and $l_i'$ are $(-1)$-curves.
	Then, we have
	\begin{equation*}
	 \xymatrix{\PP(L\otimes E_{C,0}^\vee) \ar[d]\ar[r]^(.4){\psi'}& \PP(\Forg(\psi_*(L\otimes E_{C,0}^\vee)))=:S_C\ar[d]\\ \widehat C\ar[r]^{\psi}& C,
	}
	\end{equation*}
	where the map $\psi'$, factors through $\pi^{-1}(C)\to \PP(\psi_*(L\otimes E_{C,0}^\vee))$ that corresponds to the blow-down of 10 $(-1)$-curves, say
	$l_i$ for $i=1,\ldots 10$.
	The fact that $\Forg(\psi_*(L\otimes E_{C,0}^\vee))=j_{C_*}(\O_{\PP^1}(1)\oplus\O_{\PP^1}(-1))$ implies that $\PP(\Forg(\psi_*(L\otimes E_{C,0}^\vee)))$ is isomorphic to the $2^{\rm nd}$ Hirzebruch surface, so we have a section $c$ of $\pi^{-1}(C)\to C$, such that $c^2=-2$.
	Recall that the canonical bundle of $\uY$ is $\O_{\uY}(-H-h)=\O_{\uY}(-D-S_C)$, where $D$ is the exceptional divisor.
	The adjunction formula gives
	$K_{S_C} \equiv -\res{D}{S_C}$.
	By the adjunction formula $c^2 = -2$, implies that $D\cdot c = 0$.
	Hence, we can see $c$ as a rational curve in $Y$ of degree $2$.
	The space of conics in $Y$ is four dimensional, but by Theorem \ref{thm:Md}, $\MMM_2$ has dimension 5.

	\smallskip

	\noindent {\bf Case 2:} {\em Rank $d=3$.}
	Let $G$ be supported in $C\in \abs{\O_{\PP^2}(3)}$.
	Let $C$ be smooth and intersect $\Delta$ transversally.
	Note that $\res{\Omega_{\PP^2}(2h)}{C}=F$ is an Atiyah bundle of degree 3, so we have to show that $H^0(C, G\otimes F^\vee)=0$.
	By \eqref{eqn:dHoms} and semi-continuity, $G$ has only three possibilities (as $\O_{\PP^2}$-module):
	\begin{align}
	&G \cong \text{ Atiyah bundle of degree }3\label{eqn:bo2},\\
	&G \cong (j_C)_*(\L_1\oplus\L_2)\label{eqn:dol1},\\
	&G \cong (j_C)_*(\L_0\oplus\L_3)\label{eqn:dol2},
	\end{align}
	where $j_C$ denotes the embedding and $\L_i$ are generic line bundles of degree $i$ on $C$.

	If we are in situation \eqref{eqn:bo2}, then the desired vanishing holds.

	As before, assume for a contradiction, that \eqref{eqn:dol1} holds.
	The fact that $\Forg(\psi_*(L\otimes E_{C,0}^\vee))=(j_C)_*(\L_1\oplus\L_2)$ implies that $\PP(\Forg(\psi_*(L\otimes E_{C,0}^\vee)))\to C$ has
	a section $c$ of $\pi^{-1}(C)\to C$, such that $c^2=-1$.
	Recall that the canonical bundle of $\uY$ is $\O_{\uY}(-H-h)=\O_{\uY}(-D+h-S_C)$, where $D$ is the exceptional divisor.
	The adjunction formula gives
	$K_{S_C} \equiv \res{(-D+h)}{S_C}$.
	By the adjunction formula $c^2 = -1$, implies that $D\cdot c = h\cdot c-1=2$.
	Hence, we can see $c$ as an elliptic quintic curve in $Y$ meeting the projection line $l_0$ in 2 points.
	The space of elliptic quintic curves in $Y$ is 10 dimensional \cite[Thm.~4.5]{MT}, so the one meeting $l_0$ in 2 points form an 8 dimensional family
	\cite[Lemma 4.6]{MT}, but $\MMM_3$ has dimension 10 (by Theorem \ref{thm:Md}).

	It remains to consider the last case \eqref{eqn:dol2}.
	The fact that $\Forg(\psi_*(L\otimes E_{C,0}^\vee))\cong (j_C)_*(\L_0\oplus\L_3)$ implies that $\PP(\Forg(\psi_*(L\otimes E_{C,0}^\vee)))\to C$ has
	a section $c$ of $\pi^{-1}(C)\to C$, such that $c^2=-3$.
	Recall that the canonical bundle of $\uY$ is $\O_{\uY}(-H-h)=\O_{\uY}(-D+h-S_C)$, where $D$ is the exceptional divisor.
	The adjunction formula gives
	$K_{S_C} \equiv \res{(-D+h)}{S_C}$.
	By the adjunction formula $c^2 = -3$, implies that $D\cdot c = h\cdot c-3=0$.
	Hence, we can see $c$ as an elliptic curve in $Y$ of degree $3$ (hence plane).
	The space of plane cubics in $Y$ is nine dimensional, but by Theorem \ref{thm:Md}, $\MMM_3$ has dimension 10.

	\bigskip

	\noindent {\bf Step 3:} {\em Second vanishing and stability.}
	We want to show that
	$H^1(Y,\Xi_3^{-1}(G)\otimes \O_Y(H))=0$.
	By the second part of Lemma \ref{lem:TgBundle}, we need to prove that $H^1(\PP^2,\Forg(G\otimes_{\B_0}\B_1))=0$.
	Again we can argue as in Step 1 by semi-continuity and use that $G=\Xi_3(F_d)$ satisfies the vanishing.
	Thus $\Xi_3^{-1}(G)$ is an ACM bundle.
	As observed in \cite[Sect.\ 5]{CH}, $\Xi_3^{-1}(G)$ is stable since there are no Ulrich bundles of rank $1$ on a cubic threefold.
\end{proof}

Note that the sheaves $\Xi_3^{-1}(G)$ in Proposition \ref{prop:exACM} are Ulrich.
Since they lie in $\cat{T}_Y$, the same argument of \cite[Lemma 2.4]{CH} shows that their restriction to a generic hyperplane section is again Ulrich.

To complete the non-emptiness statement of Theorem \ref{thm:main3folds}, we should prove that there are stable Ulrich bundles for all ranks $r\geq 4$.
For this we can use the same deformation argument as in the proof of \cite[Thm.\ 5.7]{CH}.

\begin{rem}\label{rmk:redone}
Note that we have also reproven that $\MMM^{sU}_r$ is smooth of dimension $r^2+1$.
Indeed, the computations $\dim\Ext^1(F,F)=r^2+1$ and $\dim\Ext^2(F,F)=0$ have already been done in Step~1 of Theorem \ref{thm:Md}.
\end{rem}

\begin{rem}\label{rmk:cased=1}
The above proof fails, for the case $d=1$, essentially only in Step 2; more precisely, the restriction $\res{\Omega_{\PP^2}(2h)}{C}$ to a line $C\subset \PP^2$ is not semistable.
\end{rem}

\section{The $d=2$ case and the instanton bundles on cubic threefolds}\label{sec:modMd2}

	In this section we will describe explicitly the wall-crossing phenomena that link the space $\MMM_2$ to the moduli space of semistable instanton sheaves on $Y$. This example, together with Section \ref{subsec:modMd1}, should motivate our expectation that the geometry of the moduli spaces $\MMM_d$ is tightly related to the one of classical geometric objects associated to cubic threefolds.

	The argument is a bit involved and thus we prefer to sketch it here for the convenience of the reader.
	First of all, we need to analyze how stability and semistability of special objects in $\Db(\PP,\B_0)$ vary in the family of stability conditions described in Lemma \ref{lem:exstab} (see Section \ref{subsec:stabrk2}).
	This is conceptually rather standard but computationally a bit involved.
	Once this is settled, one can consider instanton sheaves $E$ and look at their images under the functor $\Xi_3$.
	It turns out that they are all stable $\B_0$-modules if $E$ is locally free (see Lemma \ref{lem:inst>B0stab}).
	On the other hand special attention has to be paid to instanton sheaves $E$ which are not locally free.
	The most delicate cases are when they are extensions of ideal sheaves of two lines, one of which is the line of projection $l_0$.

	Having the toy-model of $\MMM_1$ in mind, it is rather clear that all this leads naturally to a wall-crossing phenomena.
	This will be described in Theorem \ref{thm:d=2}, where again we combine the classical description of the moduli space of semistable instanton sheaves \cite{Druel} and the machinery of (Bridgeland) stability conditions from Section \ref{subsec:stabrk2}.

As in Section \ref{subsec:modMd1}, the approach follows closely the discussion in \cite[Sect. 5]{MS}, but since the corresponding numerical class is not primitive, we need some extra arguments.

\subsection{Stability}\label{subsec:stabrk2}

We consider the stability function $Z_m$ (see Definition \ref{def:stabfunc}) and the (Bridgeland) stability condition $\sigma_m=(Z_m,\cat{A})$ (see Lemma \ref{lem:exstab}).

A (semi)stable $\B_0$-module $F\in \MMM_2$, remains $\sigma_m$-(semi)stable for all $m>m_0=\frac{\sqrt{5}}{8}$.
More precisely we have the following lemma.

\begin{lem}\label{lem:stabM2}
	Let $F\in \MMM_2$.
	\begin{enumerate}
		 \item[{\rm (a)}] If $m > m_0 = \frac{\sqrt{5}}{8}$, then $F$ is $\sigma_m$-stable or $F$ is the extension of two $\sigma_m$-stable coherent $\B_0$-modules of class $[\B_1]-[\B_0]$ (so properly $\sigma_m$-semistable).
		 \item[{\rm (b)}] If $F\in \NNN_2$ and $F$ is stable, then $F$ is $\sigma_m$-stable for all $m>\frac{1}{4}$.
			If $F$ is properly semistable, then $F$ is the extension of two $\sigma_m$-stable coherent $\B_0$-modules of class $[\B_1]-[\B_0]$ for all $m>\frac{1}{4}$.
		 \item[{\rm (c)}] Assume $m = m_0$.
		Then $F$ is $\sigma_m$-semistable and falls in one of the following cases:
		 \begin{enumerate}
			 \item[{\rm (c.i)}] $F\in \NNN_2$ is $\sigma_{m_0}$-stable.
			 \item[{\rm (c.ii)}] $F\in \NNN_2$ is properly $\sigma_{m_0}$-semistable and its JH-factors are two $\sigma_{m_0}$-stable coherent $\B_0$-modules of class $[\B_1]-[\B_0]$.
			 \item[{\rm (c.iii)}] $F \in \MMM_2\setminus \NNN_2$ and $F$ is properly $\sigma_{m_0}$-semistable and its JH-factors are $\B_0[1]$, $\B_1$ and a coherent $\B_0$-module of class of class $[\B_1]-[\B_0]$.
			 \item[{\rm (c.iv)}] $F \in \MMM_2\setminus \NNN_2$ and $F$ is properly $\sigma_{m_0}$-semistable and its JH-factors are twice $\B_0[1]$ and twice $\B_1$.
		 \end{enumerate}
	\end{enumerate}
\end{lem}

\begin{proof}
	Suppose that $0\to A \to F \to B \to 0$ destabilizes $F$ in the stability condition $\sigma_m$ for $m>\frac{1}{4}$, where $A,B\in \cat{A}$ and $A$ is $\sigma_m$-stable.
	We have
	\begin{equation*}
	0\to \H^{-1}(B)\to A \to F \to \H^0(B) \to 0
	\end{equation*}
	and $\Im Z_m([F])= 4m$.
	Note that $J_m:=\Im Z_m$ is an additive function in $K(\coh(\PP^2,\B_0))$ that takes values in $m\ZZ$.
	Moreover, by the main property of the stability function since $A,B\in \cat{A}$, we have 	$J_m(A), J_m(B)\geq 0$.
	Thus, $J_m (A)$ can only take values in $\set{0, m,2m,3m, 4m}$.
	Note also that, since $F$ and $\H^0(B)$ are torsion, $\rk(\H^{-1}(B))=\rk(A)$.

% 	Since $\H^0(B)$ is torsion, by \cite[Lem.~2.13(ii)]{BMMS}, $J_m(\H^0(B))\in \set{0,2m, 4m}$.
% % 	More precisely, if it is supported on points, then $J_m(\H^0(B))=0$, and
% % 	if it is supported on a curve, then $J_m(\H^0(B))\in \set{2m, 4m}$.
% 	We claim that $J_m(\H^0(B))=4m$ is not possible.
% 	Indeed, then the kernel of $F\to \H^0(B)$ has rank 0 and vanishing $c_1$.
% 	Since it cannot be supported on points ($F$ is locally free on its support), we would have $\H^{-1}(B)\cong A$, which is impossible since $\H^{-1}(B)\in \cat{F}$ and $A\in \cat{T}$.
% % 	If $\H^{-1}(B)\neq 0$, then $r(\H^{-1}(B))=r(A)$ and $c_1(\H^{-1}(B))=c_1(A)$, which is impossible since $\H^{-1}(B)\in \cat{F}$ and $A\in \cat{T}$.
% %	Since $A$ cannot be supported on points ($F$ is locally free on its support), then $A=0$, which is impossible.
% 	
% 	Hence, we have only two possibilities for $J_m(\H^0(B))$, namely $0$ and $2m$.
	
	Let $\ch([A])=(r,c_1,\ch_2)$.
	Observe that $\mathrm{Re}\, Z_m(F)= 0$, hence $A$ destabilizes if $\mathrm{Re}\,Z_m (A) \leq 0$.
	Since $A$ is $\sigma_m$-stable, we distinguish two cases: either $A$ is torsion or $A$ is torsion-free of rank $r=\rk(A)=\rk(\H^{-1}(B))>0$.

\medskip

	If $A$ is torsion, then $\rk(\H^{-1}(B))=0$ and since $\H^{-1}(B)\in \cat{F}$, we have $\H^{-1}(B)=0$.
	By \cite[Lemma 2.13(ii)]{BMMS}, we have that $c_1$ is even.
	As $F\in \MMM_2$ is a semistable $\B_0$-module, then $c_1\leq 2$.
	In that case, $A$ cannot be supported on points ($F$ is locally free on its support), we have $c_1=2$.
	In order to destabilize $F$ in the stability condition $\sigma_m$, we need $\mathrm{Re}\,(Z_m([A]))=rm^2-\dfrac{9r}{64}-\dfrac{c_1}{2}-\dfrac{\ch_2}{2}\leq 0$, so $\ch_2\geq -2$.
	But $F\in \MMM_2$ and since $c_1=2$ and we have $\ch_2\leq -2$.
	Thus $\ch([A])=(0,2,-2)$ and $F$ is a properly $\sigma_m$-semistable object (for all $m>\tfrac{1}{4}$) whose JH-factors have class $[\B_1]-[\B_0]$.

\medskip
	
	Suppose now that $A\in \mathbf{T}$ is torsion-free, so all its HN-factors with respect to the slope stability have slope $\mu > -1$.
	Note that $J_m(A)>0$ since $A$ cannot be supported on points.
	Moreover, if $J_m(A)= 4m$, then $J_m(B)=0$, $\phi(B)=1$, and $B$ would not destabilize $F$.
	So $J_m(A)\in \set{m,2m,3m}$.

	In order to $\sigma_m$-destabilize $F$, we need that
	\begin{equation}\label{eqn:re<=0}
	\mathrm{Re}(Z_m([A]))=rm^2-\dfrac{9r}{64}-\dfrac{c_1}{2}-\dfrac{\ch_2}{2}=\frac{1}{r}\left(-\chi(A,A)+m^2r^2-\frac{1}{4}\left(r+c_1\right)^2\right)\leq 0.
	\end{equation}
	Moreover, we can assume that $A$ is $\mu$-stable.	
	Then, by Lemma \ref{lem:Columbus}, we have $-1\leq-\chi(A,A)$.
	Since $m\leq J_m(A)\leq 3m$, we have  $-\frac{9}{4}\leq -\frac{1}{4}\left(r+c_1\right)^2$.
	Thus, from \eqref{eqn:re<=0} and the previous inequalities we deduce $m^2r^2-\frac{13}{4}\leq 0$.
	Since $m>\frac{1}{4}$ and $r\in 4\NN_{>0}$, this implies $r=4$.

	\medskip

% 	Suppose first that $J_m(\H^0(B))= 2m$.
	Now we go through a case by case study depending on $J_m(A)$.

	\smallskip

	\noindent{\bf Case $J_m(A)=m$.}
	In this case $c_1=-3$ and $r=4$, so \eqref{eqn:re<=0} becomes $-\chi(A, A) + 16m^2 -\tfrac{1}{4}\leq  0$.
	Since $m>\tfrac{1}{4}$ and $-\chi(A, A)=\hom^1(A,A)-1$, we have $\hom^1(A,A)-\frac{1}{4}<0$, so $A$ is rigid.
	Moreover $\ch_2=\frac{5}{2}$, by \eqref{eqn:chiB0}.
	Hence, $A\cong \B_1$ and $[B]=[\B_{1}]-2[\B_{0}]$.
	Moreover, since $16m^2-\frac{5}{4}\leq 0$, then $m\leq \frac{\sqrt{5}}{8}$.

	\smallskip

	\noindent{\bf Case $J_m(A)=2m$.} Under this assumption $c_1=-2$, which is impossible by \cite[Lemma 2.13(ii)]{BMMS}.

	\smallskip

	\noindent{\bf Case $J_m(A)=3m$.} Here $c_1=-1$ and $r=4$, so \eqref{eqn:re<=0} becomes $-\chi(A, A) + 16m^2 -\tfrac{9}{4}\leq  0$.
	Since $m>\tfrac{1}{4}$ and $-\chi(A, A)=\hom^1(A,A)-1$, we have $\hom^1(A,A)-\frac{9}{4}<0$, which implies $\hom^1(A,A)\leq 2$.
	\begin{enumerate}
		\item[{\rm (a)}] If $A$ is rigid, then $\ch_2=\frac{3}{2}$, by \eqref{eqn:chiB0}.
		Hence $A\cong \B_2$, which is impossible since $0\neq \hom(\B_2,F)= h^0(\PP^2,F(-h))$ contradicts the Gieseker semistability of $F\in \MMM_2$.

		\item[{\rm (b)}] If $\chi(A,A)=0$, then $\ch_2=1$, by \eqref{eqn:chiB0}.
		This implies that $[A]\not \in \ZZ[\B_{-1}]\oplus\ZZ[\B_{0}]\oplus\ZZ[\B_{1}]$, which contradicts \cite[Prop.~2.12]{BMMS}.

		\item[{\rm (c)}] If $\chi(A,A)=-1$, then $\ch_2=\frac{1}{2}$.
		Hence $[A]=-[\B_0]+2[\B_1]$, so $B\cong \B_0[1]$.
		In particular, $\H^{0}(B)=0$.
		Moreover, $16m^2-\frac{5}{4}\leq 0$, so $m\leq \frac{\sqrt{5}}{8}$.
	\end{enumerate}

	\smallskip

	Summarizing, if $F\in \NNN_2$ is stable then it is $\sigma_m$-stable for all $m>\frac{1}{4}$.
	If $F\in \NNN_2$ is properly semistable, then its two JH-factors are $\sigma_m$-stable for all $m>\frac{1}{4}$.
	If $F\in \MMM_2\setminus \NNN_2$, since $\chi(\B_1,F)=0$ we have
	\begin{equation*}
		\hom(\B_1,F)=\hom^1(\B_1,F)=\hom^1(F,\B_0)=\hom(F,\B_0[1]),
	\end{equation*}	
	and $F$ admits a morphism from $\B_1$ and it has also a morphism to $\B_0[1]$.
	Hence it could be $J_m(A)\in\set{m,3m}$.
	
	We study more precisely these two cases.
	If $J_m(A)=m$, then we claim there exist the following exact sequences in $\cat{A}$:
	\begin{equation}\label{eqn:aggq}
		\begin{split}
		 0\to \B_1 \to F \to C\to 0, & \qquad\text{where }[C]=[\B_1]-2[\B_0],\\
		 0\to  C' \to F \to \B_0[1] \to 0, & \qquad\text{where }[C']=2[\B_1]-[\B_0].
		\end{split}
		%it exists also a C' such that the exact sequence is in \text{ and }\B_0-\cat{mod}.
	\end{equation}
	The second exact sequence is obtained from the first one using $\chi(\B_1,F)=0$, so $\hom_{\cat{A}}(F,\B_0[1])\neq 0$.
	Indeed, it remains to prove that $F \to \B_0[1]$ needs to be surjective in $\cat{A}$.
	If not, let $L$ be the cokernel.
	Clearly $\H^0(L)=0$.
	Then $L=L'[1]$, where $L'$ is a torsion-free $\B_0$-module in $\cat{F}$.
	Let $T:=\mathrm{Im}_{\cat{A}}(F\to B_0[1])$.
	Note that $\H^0(T)$ is a torsion sheaf.
	Then, if $L'\neq 0$, then $\B_0\to L'$ needs to be injective and $\H^{-1}(T)=0$.
	Therefore $T\in \cat{T}$ and it is a quotient of $F$ as a $\B_0$-module.
	We know that $F$ is a Gieseker semistable $\B_0$-module and $c_1(T)\geq 2$.
	This implies $c_1(L')\geq -3$, which contradicts $L'\in \cat{F}$, since by \cite[Lemma 2.13(i)]{BMMS} $\rk(L')\geq 4$.

	Equivalently, if we are in case $J_m(A)=3m$, then we claim that we get again the exact sequences \eqref{eqn:aggq}.
	Indeed, now the first exact sequence is obtained from the second one by using $\chi(\B_1,F)=0$, so $\hom_{\cat{A}}(\B_1,F)\neq 0$.
	In that case we need to prove that $B_1\to F$ is injective in $\cat{A}$.
	If not, let $K$ be the kernel.
	Clearly $\H^{-1}(K)=0$.
	Then $K$ is a $\B_0$-module in $\cat{T}$.
	If $K\neq 0$, then $K\to \B_1$ needs to be injective. %it cannot be zero if not $\H^{-1}(T)=\H^0(K)$, one in $\cat{T}$ and the other in $\cat{F}$
	Hence $T:=\mathrm{Im}_{\cat{A}}(\B_1\to F)\in \cat{T}$ and it is a subobject of $F$ as a $\B_0$-module.
	We know that $F$ is a Gieseker semistable $\B_0$-module and $c_1(T)\leq 2$.
	This implies $c_1(K)\leq -5$, which contradicts $K\in \cat{T}$, since by \cite[Lemma 2.13(i)]{BMMS} $\rk(K)\geq 4$.
	
	\smallskip
	
	In both cases we can summarize the situation in the following commutative diagram of exact sequences of $\sigma_{m_0}$-semistable objects in $\cat{A}$
	\begin{equation}\label{eqn:diagrimportant}
			\xymatrix{&&0\ar[d]&0\ar[d]\\
				0\ar[r]&\B_1\ar[r]\ar@{=}[d] & C'\ar[d]\ar[r]&D\ar[d]\ar[r]&0\\
				0\ar[r]&\B_1\ar[r]& F\ar[d]\ar[r] &C\ar[d]\ar[r]&0\\
				&&\B_{0}[1]\ar@{=}[r]\ar[d]& \B_{0}[1]\ar[d] \\ %intersection point 1
				&&0&0
			}
	\end{equation}
	
	On the one hand, by the mid vertical exact sequence, we have that $\H^{-1}(C')=0$.
	Hence, by the top horizontal exact sequence, $\H^{-1}(D)=0$ or $\H^{-1}(D)=\B_1$.
	Note that $\H^{-1}(D)=\B_1\not\in \cat{F}$, so $D\in \coh(\PP^2,\B_0)$.
	
	Summing up, when $m=m_0$, we have in those cases that $\B_0[1]$ and $\B_1$ are two JH-factors of $F$.
	Moreover, the remaining part $D$ of the JH-filtration is a $\sigma_{m_0}$-semistable $\B_0$-module, with class $[\B_1]-[\B_0]$.
	
	When $C\in \gen{\B_1}^{\perp}$ (equivalently $D\in \NNN_1$), then the we fall in case (c.iii) and the JH-factors for $m=m_0$ are $\B_0[1]$, $\B_1$ and $D$.
	When $C\not\in \gen{\B_1}^{\perp}$ (equivalently $D\in \MMM_1\setminus\NNN_1$), we fall in case (c.iv).
\end{proof}

Let $w$ be the numerical class $2[\B_1]-2[\B_0]$.
As a consequence of the previous lemma we get that $\MMM_2$ embeds inside $\MMM^{\sigma_m}(\PP^2,\B_0;w)$ when $m > m_0$.
With the aim of proving the other inclusion, we need the following lemma which adapts \cite[Lemma 5.8]{MS} to our situation.

\begin{lem}\label{lem:keepstab}
	Let $m_1 > m_0= \frac{\sqrt{5}}{8}$ and let $G$ be a $\sigma_{m_1}$-(semi)stable object with numerical class $w$.
	Then $G$ is $\sigma_m$-(semi)stable, for all $m > m_0$, and $\sigma_{m_0}$-semistable.
\end{lem}

\begin{proof}
	Assume, for a contradiction, that $G$ is not $\sigma_m$-semistable (resp.~$\sigma_m$-stable) at $m \geq m_0$.
	Then we have an exact sequence in $\cat{A}$
	\begin{equation*}
	0 \to A \to G \to B \to 0,
	\end{equation*}
	where $A \neq 0$ is $\sigma_m$-stable and ${\rm Re}(Z_m (A)) < 0$ (resp.~$\leq 0$).
	Let $\ch(A) = (r, c_1, \ch_2)$.
	The same argument as in \cite[Lemma 5.8]{MS} shows that $r\neq 0$ and $\Im(Z_m(A))\in \set{ m,2m,3m}$.
	But then, the same casuistry as in Lemma \ref{lem:stabM2} shows that this can only happen when $m < m_0$ (resp.~$m < m_0$ or $m=m_0$ and $G\not\in \langle\B_1\rangle^{\perp}$).
\end{proof}

Now, let $G\in \MMM^{\sigma_{m_1}}(\PP^2,\B_0;w)$ be a (semi)stable object for $m_1\geq m_0$.
By Lemma \ref{lem:keepstab}, we have two possibilities: either $G$ is $\sigma_{m}$-(semi)stable for all $m\geq m_0$, or $m_1=m_0$ and it either stabilizes or destabilizes for all $m>m_0$.
We will see in Lemma \ref{lem:countext} that it destabilizes or stabilizes depending whether $\hom(\B_1,\G)$ is maximal in its S-equivalence class or not.

\begin{lem}\label{lem:smstab-stab}
	Let $G \in\cat{A}$ be a $\sigma_m$-(semi)stable object, for all $m \geq m_0$, with numerical class $w$.
	Then $G$ is a (semi)stable $\B_0$-module pure of dimension $1$.
\end{lem}

\begin{proof}
	We argue as in \cite[Lemma 5.9]{MS} to deduce that $G$ is a pure $\B_0$-module of dimension $1$.
	If $A$ is a stable $\B_0$-module that destabilizes $G$, then $\mathrm{Re}\,Z_m(A)<0$ (resp.~$\leq 0$) so $G$ would not be $\sigma_m$-semistable.
\end{proof}

As a straightforward consequence of the previous lemmas, we get the following.

\begin{cor} \label{cor:M2vsMsig}
	Let $w=2[\B_1]-2[\B_0]$.
	Then $\MMM_2=\MMM^{\sigma_{m}}(\PP^2,\B_0;w)$, for all $m>m_0=\frac{\sqrt{5}}{8}$.
\end{cor}

	Finally, we study in general the S-equivalence classes in $\MMM^{\sigma_{m_0}}(\PP^2,\B_0;w)$ which contain objects outside $\gen{\B_1}^{\perp}$.
	In particular, we will study the S-equivalence classes of the objects $F\in\MMM_2$, which become $\sigma_{m_0}$-semistable with JH-factors as in cases (c.iii) and (c.iv) of Lemma \ref{lem:stabM2}.
	The following lemma will be useful in the next section to prove Theorem \ref{thm:d=2}.

\begin{lem}\label{lem:countext}
	Let $w=2[\B_1]-2[\B_0]$.
	Given $G$ an object in $\MMM^{\sigma_{m_0}}(\PP^2,\B_0;w)\setminus \gen{\B_1}^{\perp}$. Then, it falls in one of the following cases:
	\begin{enumerate}
		\item[{\rm (a)}] $G$ is in the S-equivalence class of $\B_0[1]\oplus \B_1\oplus \Xi_3(\I_l)$, with $l\neq l_0$.
		The indecomposable objects in this S-equivalence class in $\MMM^{\sigma_{m_0}}(\PP^2,\B_0;w)$ are represented by:
		\begin{itemize}
			\item[{\rm (a.i)}] Gieseker semistable $\B_0$-modules in $\MMM_2\setminus\NNN_2$ that are parametrized by a $\PP^2$;
			\item[{\rm (a.ii)}] Gieseker properly semistable $\B_0$-modules in $\MMM_2\setminus\NNN_2$ that are parametrized by a $\PP^1$ contained in the $\PP^2$ above; in the complement $\PP^1$ inside $\PP^2$, the $\B_0$-modules are Gieseker stable;
			\item[{\rm (a.iii)}] an extension of $\Xi_3(\I_{l_0})$ and $\Xi_3(\I_l)$, which lies in $\gen{\B_1}^\perp$.
		\end{itemize}
		\item[{\rm (b)}] $G$ is in the S-equivalence class of $\B_0^{\oplus 2}[1]\oplus \B_1^{\oplus 2}$.
		The indecomposable objects in this S-equivalence class in $\MMM^{\sigma_{m_0}}(\PP^2,\B_0;w)$ are represented by:
		\begin{itemize}
			\item[{\rm (b.i)}] Gieseker properly semistable $\B_0$-modules $G\in \MMM_2\setminus \NNN_2$;
			they have $\hom(\B_1,G)=2$, their S-equivalence classes as $\B_0$-modules are parametrized by a $\PP^2$, and each S-equivalence class is $\CC^2$;
			\item[{\rm (b.ii)}] indecomposable extensions between $\Xi_3(\I_{l_0})$ with itself, which are then in $\gen{\B_1}^\perp$;
			\item[{\rm (b.iii)}] objects $G$ such that $\hom(\B_1,G)=1$.
		\end{itemize}
	\end{enumerate}
	These are the only S-equivalence classes that contain $\sigma_{m_0}$-semistable objects that get properly destabilized for $m>m_0$ and $m<m_0$.
\end{lem}

\begin{proof}
	If $G\not \in \gen{\B_1}^{\perp}$, then $\B_1\to G$ is necessarily an injection in $\cat{A}$.
	Indeed, let $T:=\mathrm{Im}_{\cat{A}}(\B_1\to F)$.
	Since $\B_1$ is $\sigma_{m_0}$-stable, with $\mathrm{Re}\,(Z_{m_0}(\B_1))=0$, if $T\not\cong \B_1$, then $\mathrm{Re}\,( Z_{m_0}(T))>0$. This contradicts the semistability of $G$.
	Thus, $\B_1$ is a JH-factor of $G$.
	Since $\chi(\B_1,G)=0$, we have $\hom(\B_1,G)=\hom(G,\B_0[1])$.
	The same argument shows that $G\to \B_0[1]$ is necessarily a surjection in $\cat{A}$.
	Thus, $\B_0[1]$ is another JH-factor of $G$.
	
	Hence, as in the proof of Lemma \ref{lem:stabM2}, $G$ necessarily sits in the commutative diagram \eqref{eqn:diagrimportant}.
	Note that $D\in \cat{A}$, but unlike in Lemma \ref{lem:stabM2}, $D$ is not necessarily in $\coh(\PP^2,\B_0)$, because we have not assumed that $G$ is in $\coh(\PP^2,\B_0)$.
	
	If $D\in \gen{\B_1}^{\perp}$, then $D\cong\Xi_3(\I_l)$ for some line $l$, by \cite[Thm.~4.1]{BMMS}.
	Note that $D\in \coh(\PP^2,\B_0)$ if and only if $l\neq l_0$ and then, we are in case (a).
	If $D\not\in \gen{\B_1}^{\perp}$ or $l=l_0$, then $D$ is still properly $\sigma_{m_0}$-semistable, with JH-factors $\B_0[1]$ and $\B_1$ and we are in case (b).

	\smallskip
	
	Suppose we are in case (a) and let $G$ be a representative in the S-equivalence class such that $\hom(\B_1,G)\neq 0$.
	Note that an element in $\Hom^1(\B_0[1],\B_1)$ corresponds to an element in the projective line $\PP^1$ which is the exceptional locus of the map $\MMM_1\to F(Y)$ described in Proposition \ref{prop:d=1}.
	Taking the unique non-trivial extension of one of this $\B_0$-modules with $\Xi_3(\I_l)$ we obtain a $\PP^1$ of properly semistable $\B_0$-modules (this is (a.ii)).
	
	Now we start with an element in $\Hom^1(\Xi_3(\I_l),\B_1)$.
	By Example \ref{ex:RestrLines}, we have that
	\[
    \hom^1(\Xi_3(\I_l),\B_1)=h^1(\PP^1,\O_{\PP^1}(-1)\oplus \O_{\PP^1}(-2))=1.
    \]
	Let $C'\in \Hom^1(\Xi_3(\I_l),\B_1)$. Clearly $C'\in \coh(\PP^2,\B_0)$ since $\Xi_3(\I_l)$ and $\B_1$ are also in $\coh(\PP^2,\B_0)$.
	
	Note that $\hom^1(\B_0[1],C')=3$, because
	\begin{equation*}\label{eqn:extimp}
		0\to \Hom(\B_0,\B_1)\to \Hom^1(\B_0[1],C')\to \Hom(\B_0,\Xi_3(\I_l))\to 0.
	\end{equation*}
	Let $G\in \Hom^1(\B_0[1],C')$.
	We want to see that $G$ is a $\B_0$-module.
	We have
	\begin{equation*}
		0\to \H^{-1}(G)\to \B_0\to C'\to \H^0(G)\to 0.
	\end{equation*}
	Since $\B_0$ is torsion-free and $\rk(C')=\rk(\B_0)$, either $\B_0\to\C'$ is zero, or $\H^{-1}(G)=0$.
	Hence, the non-trivial extensions between $\B_0[1]$ and $C'$ are $\B_0$-modules and they are parametrized by a $\PP^2$.
	When the first extension is trivial, i.e., $C'=\B_1\oplus\Xi_3(\I_l)$, we recover the previous case.

	Finally, we want to see that these extensions $G$ are Gieseker semistable $\B_0$-modules.
	Since $G$ is $\sigma_{m_0}$-semistable, up to choosing $\eps$ small enough, $G$ is $\sigma_{m}$-semistable, for all $m \in (m_0 , m_0+\eps )$.
	Indeed, if not, by \cite[Prop.~9.3]{Br1}, the HN-factors of $G$ in the stability condition $\sigma_m$,
	for $m \in (m_0, m_0+\eps )$, would survive in the stability condition $\sigma_{m_0}$.
	This would contradict the $\sigma_{m_0}$-semistability of $G$.
	Since we have seen that $\MMM_2=\MMM^{\sigma_{m}}(\PP^2,\B_0;w)$ for all $m>m_0$, we get that $G \in \MMM_2$ and thus (a.i).
	If $G$ is properly semistable, then $G$ is the extension of two stable $\B_0$-modules, $G_1$ and $G_2$.
	Since $\hom(\B_1,G)\neq0$, we can suppose that $G_1\in \NNN_1$ and $G_2\in \MMM_1\setminus\NNN_1$ and we are in the aforementioned $\PP^1$.

\smallskip

	Now, suppose we are in case (a) and let $G$ be a representative in the S-equivalence class in $\gen{\B_1}^{\perp}$.
	Since $\hom^1(\B_1,\Xi_3(\I_l))=0$, we need to start with an element in $\Hom^1(\B_1,\B_0[1])$.
	By \cite[Ex.~2.11]{BMMS}, the only non-trivial extension in $\Hom^1(\B_1,\B_0[1])$ is $\Xi_3(\I_{l_0})$ and we get (a.iii).
	Thus we conclude the analysis of case (a).
	
\smallskip

	Suppose we are in case (b) and let $G$ be a representative in the S-equivalence class such that $\hom(\B_1,G)=2$.
	By the same argument as before, an extension $C$ of $\B_1$ with itself needs to be a subobject of $G$ in $\cat{A}$ while an extension $C'$ of $\B_0[1]$ with itself is a quotient of $G$ in $\cat{A}$.
	Note that necessarily $C=\B_1^{\oplus 2}$ and $C'=\B_0^{\oplus 2}[1]$.
	Hence we consider an element in $G\in \Hom^1(\B_0^{\oplus 2}[1],\B_1^{\oplus 2})$.
	
	Equivalently we can construct $G$ as the extension of two sheaves $G_1$ and $G_2$ in the exceptional locus of the map $\MMM_1\to F(Y)$ described in Proposition \ref{prop:d=1}.
	Each of them is parametrized by a $\PP^1$.
	But since the role of $G_1$ and $G_2$ is symmetric, we obtain that the S-equivalence classes of the $G$ as objects in $\MMM_2$ are parameterized by $\PP^1\times\PP^1$ quotiented by the natural involution.
	Thus, the S-equivalence classes of the $G$'s are parameterized by a $\PP^2$ and we obtain (b.i). Note that $\Ext^1(G_1,G_2)\cong\CC^2$.
	
    \smallskip

	Let $G$ be in $\gen{\B_1}^{\perp}$ and as in (b).
	Again, $G$ is obtained from an element in $\Hom^1(\B_1^{\oplus 2},\B_0^{\oplus 2}[1])$.
	Equivalently we can construct $G$ as the extension of the two unique non-trivial extensions in $\Hom^1(\B_1,\B_0[1])$.
	Each of them is $\Xi_3(\I_{l_0})$ and $\Ext^1(\Xi_3(\I_{l_0}),\Xi_3(\I_{l_0}))\cong\CC^2$. This is (b.ii)
	
	\smallskip
	
	The remaining indecomposable objects $G$ in case (b) have $\hom(\B_1,G)=1$ (as in (b.iii))
	and the last statement of the lemma follows from the fact that these are the only S-equivalence classes that contain the objects $G$ such that $\hom(\B_1,G)\neq 0$ and the objects $G'$ such that $\hom(G',\B_1)\neq 0$.
\end{proof}

\begin{rem}\label{rmk:fibcontr}
	Notice that the S-equivalence classes in (a) and (b) contain the $\B_0$-modules in cases (c.iii) and (c.iv) of Lemma \ref{lem:stabM2} respectively.
	Moreover, from Lemma \ref{lem:countext} and \cite[Prop.~9.3]{Br1}, it follows that a $\sigma_{m_0}$-semistable object $G$ remains semistable for $\sigma_m$, with $m>m_0$ if and only if $\hom(\B_1,G)$ is maximal in its S-equivalence class. This happens in cases (a.i) and (b.i) of the previous Lemma. Under these circumstances, two objects $G_1$ and $G_2$ in the same S-equivalence class in $\sigma_{m_0}$ belong to different S-equivalence classes in $\sigma_m$, for $m>m_0$. For this one uses that $G_1$ and $G_2$ are Gieseker (semi)stable and invokes Corollary \ref{cor:M2vsMsig}.
	On the other hand, $G$ remains semistable for $\sigma_m$, with $m<m_0$ if $\hom(\B_1,G)=0$. This happens in cases (a.iii) and (b.ii) of Lemma \ref{lem:countext}. It is clear that if $G$ is as in (b.iii), then $G$ is not $\sigma_m$-semistable for $m>m_0$ or $\frac{1}{4}<m<m_0$.	
\end{rem}

\subsection{Instanton sheaves}\label{subsection:instanton}

Now we want to give a geometric interpretation of $\MMM^{\sigma_{m}}(\PP^2,\B_0;w)$ for $m\leq m_0$.
The appropriate objects are the instanton sheaves.

\begin{defn}\label{def:instantons}
	We say that $E\in \coh(Y)$ is an \emph{instanton sheaf} if $E$ is a Gieseker semistable sheaf of rank $2$ and Chern classes
	$c_1(E) = 0$ and $c_2 (E) = 2$.
	When $E$ is locally free, we call it \emph{instanton bundle}.
\end{defn}

 An instanton sheaf according to the above definition would be called \emph{instanton sheaf of charge $2$} in the existing literature. In general, an \emph{instanton bundle of charge $s\geq 2$} is a locally free sheaf $E$ of rank $2$, Chern classes
	$c_1(E) = 0$ and $c_2 (E) = s$, and such that $H^1(Y,E(-1))=0$ (see, for example, \cite[Def.~2.4]{Kuz:V14}). It is easy to show that if the charge is minimal (i.e., $c_2(E)=2$), then the condition $H^1(Y,E(-1))=0$ is automatically satisfied (see \cite[Cor.~3.3]{KuzNew}).

\begin{rem}\label{rmk:nlf_inst}
By \cite[Thm.~3.5]{Druel}, each semistable instanton sheaf falls under one of the following cases:
\begin{enumerate}
\item[(1)] $E$ is stable and locally free.
\item[(2)] $E$ is stable but not locally free. In this case, $E$ is obtained by the construction in Example \ref{ex:objMd}. In fact, these are the only stable instanton sheaves that are not locally free.
\item[(3)] $E$ is properly semistable. In this situation, $E$ is extension of two ideal sheaves of lines in $Y$.
\end{enumerate}

Moreover, given a stable instanton bundle $E$, then $E(1)$ is globally generated \cite[Thm. 2.4]{Druel}, so $E$ is an Ulrich bundle.
Indeed, $E$ is associated to a non-degenerate smooth elliptic quintic $C$ via the Serre construction (see \cite[Cor.~2.6]{Druel} and compare it with Lemma \ref{lem:Ulr_geom}).
\end{rem}

The following will be crucial in our analysis.

\begin{lem}\label{lem:inst>B0stab}
	Let $E$ be a stable instanton bundle.
	Then $\Xi_3(E)$ is a stable $\B_0$-module.
\end{lem}

\begin{proof}
	Let $F$ be a stable instanton bundle of minimal charge.
	By \cite[Cor. 2.6]{Druel} a stable Ulrich bundle $F$ of rank 2 is associated to a non-degenerate smooth elliptic quintic $C$ via the Serre construction
	\begin{equation}\label{eqn:serre1}
		0\to \O_Y(-H) \to F \to \I_C(H)\to 0.
	\end{equation}

	Note that $\Psi(\sigma^*\O_Y(-H))=\B_{-1}$ and $\Psi(\sigma^*\O_Y(H))=\B_2[1]$.
	Applying the functor $\Psi\circ \sigma^*$ to the short exact sequence
	\[
	0\to\I_C(H)\to\O_Y(H)\to\O_C(H)\to 0,
	\]
	we get
	\begin{equation}\label{eqn:defnf}
	0\to \H^{-1}(\Psi(\sigma^*\I_C(H)))\to \B_2 \stackrel{f}\to \Psi(\sigma^*\O_C(H)) \to \H^{0}(\Psi(\sigma^*\I_C(H)))\to 0.
	\end{equation}
	On the other hand, from \eqref{eqn:serre1} we obtain
	\begin{equation}\label{eqn:defng}
	0\to \H^{-1}(\Psi(\sigma^*F))\to \H^{-1}(\Psi(\sigma^*\I_C(H))) \stackrel{g}\to {\B_{-1}} \to \H^{0}(\Psi(\sigma^*F)) \to \H^{0}(\Psi(\sigma^*\I_C(H)))\to 0.
	\end{equation}
	Observe that $\H^{-1}(\Psi(\sigma^*\I_C(H)))\subseteq \B_2$ is a non-trivial torsion-free sheaf of rank $4$.
	Hence, the map $g$ is either injective or zero.

\bigskip

\noindent \textbf{Step 1:} Assume that the associated elliptic quintic $C$ does not intersect $l_0$.
	Since $C\cap l_0=\emptyset$, we have $\Psi(\sigma^*\O_C(H))= F_{\uC}[1]$, where $F_{\uC}$ is a rank $2$ torsion-free supported on the irreducible curve $\uC= \pi( \sigma^{-1}(C))\subset \PP^2$.
	Hence, \eqref{eqn:defnf} becomes
	\begin{equation*}
	0\to \H^{-1}(\Psi(\sigma^*\I_C(H)))\to \B_2 \stackrel{f}\to F_{\uC} \to \H^{0}(\Psi(\sigma^*\I_C(H)))\to 0.
	\end{equation*}
	On the one hand, note that $f$ could be surjective, or zero, or have cokernel supported on points. Indeed, by \cite[Lemma 2.13 (ii)]{BMMS}, the image of $f$ is either supported on points or trivial or it is a rank $2$ torsion-free subsheaf of $F_{\uC}$. As $F_{\uC}$ is torsion-free, the first possibility cannot be realized. Thus $f$ has to be as we claimed above.

	Now we observe that $g$ in \eqref{eqn:defng} is injective.
	Assume, by contradiction that $g$ is zero.
	Hence, we have the following exact sequence
	\begin{equation*}
	0\to {\B_{-1}} \to \H^{0}(\Psi(\sigma^*F)) \to \H^{0}(\Psi(\sigma^*\I_C(H)))\to 0.
	\end{equation*}
	If $\H^{0}(\Psi(\sigma^*\I_C(H)))$ is supported at most in dimension 0, then we have $0\neq \Hom^2(\B_1,\B_{-1})\hookrightarrow \Hom^2(\B_1,\H^0(\Psi(F)))$.
	Note that we have an exact triangle $$\H^{-1}(\Psi(\sigma^* F))[1]\to \Psi(\sigma^* F)\to \H^0(\Psi(\sigma^* F)),$$ so this would imply $\Hom^2(\B_1,\Psi(\sigma^* F))\neq 0$.
	But $F\in \cat{T}_Y$, so $\Psi(\sigma^* F)\in\langle \B_1\rangle^\perp$ and we get a contradiction.
	If $f$ and $g$ are zero, then $\H^{-1}(\Psi(\sigma^* F))=\B_2$ and we get a contradiction because $\Hom^0(\B_1,\H^{-1}(\Psi(\sigma^*F)))\neq 0$. The case when $f$ is surjective and $g$ is trivial can be excluded by a similar argument as we would have $\B_{-1}\cong \H^{0}(\Psi(\sigma^*F))$.

	Therefore, $g$ is injective and $\Psi(\sigma^*F)$ is a torsion sheaf with class $2[\B_1]-2[\B_0]$.

\bigskip

\noindent \textbf{Step 2:} Assume that the associated elliptic quintic $C$ intersects $l_0$ transversally in a point.
	Since $C\cap l_0=\set{p}$, we have $\Psi(\sigma^*\O_C(H))= \Psi(\O_{C'\cup \gamma}(H))$, where by abuse of notation we denote by $C'$ the strict transform of $C$ and $\gamma\subset D$ is the line $\sigma^{-1}(p)$. Hence, \eqref{eqn:defnf} becomes
	\begin{equation}\label{eqn:fe1}
	0\to \H^{-1}(\Psi(\sigma^*\I_C(H)))\to \B_2 \stackrel{f}\to \Psi(\O_{C'\cup \gamma}(H)) \to \H^{0}(\Psi(\sigma^*\I_C(H)))\to 0.
	\end{equation}
	
To characterize $\Psi(\O_{C'\cup \gamma}(H))$ better, consider the exact sequence on $\coh(\uY)$
	\begin{equation}\label{eqn:popline}
		0\to \I_{p,C'}(H)\to \O_{C'\cup \gamma}(H) \to \O_{\gamma}(H)\to 0.
	\end{equation}
On the one hand, we need to compute $\Psi(\O_{\gamma}(H))$. As $\gamma\subset D$ it makes sense to consider the ideal sheaf $\I_{\gamma,D}$ which is actually equal to $\I_{\sigma^{-1}(p),\sigma^{-1}(l_0)}=\sigma^*\I_{p,l_0}=\sigma^*(\O_{l_0}(-H))=\O_D(-H)$. Now, tensoring the exact sequence
		\begin{equation*}
			0 \to  \O_D(-H)\to     \O_D \to \O_{\gamma}\to 0
		\end{equation*}
by $D=H-h$, we have
		\begin{equation}\label{eqn:fe2}
			0 \to  \O_D(-h)\to     \O_D(D) \to \O_{\gamma}(H-h)\to 0.
		\end{equation}
By \cite[Ex. 2.11]{BMMS} and applying the functor $\Psi$, it provides the exact triangle
		\begin{equation*}
			 \B_{-1}[1]\to     \B_0[1] \to \Psi(\O_{\gamma}(H-h)).
		\end{equation*}
By construction, we know that $\Psi(\O_{\gamma}(H-h))$ is a torsion sheaf in degree $-1$ and we have the following exact sequence
		\begin{equation*}
			0\to \B_{-1}\to     \B_0 \to \H^{-1}(\Psi(\O_{\gamma}(H-h)))\to 0.
		\end{equation*}
By definition $\Psi(F\otimes \O_{\uY}(mh))=\pi_*(F\otimes \O_{\uY}(mh)\otimes \E\otimes \O_{\uY}(h))[1]=\Psi(F)\otimes \O_{\uY}(mh)$.
		Hence, if we tensor \eqref{eqn:fe2} by $\O_{\uY}(h)$ and we apply $\Psi$ again, we get that $\Psi(\O_{\gamma}(H))$ is a torsion sheaf in degree $-1$ and the following exact sequence
		\begin{equation}\label{eqn:psigammaH}
			0\to \B_{1}\to     \B_2 \to \H^{-1}(\Psi(\O_{\gamma}(H)))\to 0.
		\end{equation}
	Note that, $\Hom(\B_2,\H^{-1}(\Psi(\O_{\gamma}(H))))\cong\CC$.

Using the above discussion and \eqref{eqn:popline}, we get the commutative diagram
\begin{equation}\label{eqn:fediag}
	\xymatrix{&&0\ar[d]&0\ar[d]\\
		&0\ar[r]&\H^{-1}(\Psi(\sigma^*\I_C(H)))\ar[r]\ar[d]& \B_1\ar[d]\\ % curved arrow 1
		&&\B_{2}\ar@{=}[r]\ar[d]^f& \B_2\ar[d] \\ %intersection point 1
		0\ar[r]&\H^{-1}(\Psi(\I_{p,C'}(H)))\ar[r]\ar@{=}[d]
		%intersection point 2
		&\H^{-1}(\Psi(\O_{C'\cup \gamma}(H)))\ar[d]\ar[r]^-{h}&\H^{-1}(\Psi(\O_{\gamma}(H)))\ar[d]\ar[r]&0\\
		&\H^{-1}(\Psi(\I_{p,C'}(H)))\ar[r]& \H^{0}(\Psi(\sigma^*\I_C(H)))\ar[d] &0\\
		&&0&
	}
\end{equation}
Furthermore, we have $\Psi(\I_{p,C'}(H))\cong F_{\widetilde{C'}}[1]$, where $F_{\widetilde{C'}}$ is a rank $2$ torsion-free sheaf supported on $\widetilde{C'}= \pi(C')\subset \PP^2$ which is irreducible. Indeed, by definition, $\Psi(\I_{p,C'}(H))=\pi_*(\I_{p,C'}(H)\otimes \E\otimes \O_{\uY}(h))[1]$.
Since the fibers of $\pi$ restricted to $C'$ are only points or empty we have
$\Psi(\I_{p,C'}(H))=R^0\pi_*(\I_{p,C'}(H)\otimes \E\otimes \O_{\uY}(h))[1]$, where $R^0\pi_*(\I_{p,C'}(H)\otimes \E\otimes \O_{\uY}(h))$ is a sheaf supported on $\widetilde{C'}=\pi(C')$. Observe that $C'\subset \uY$ is a quartic, so the image $\pi(C')$ is either a line, a conic, or a quartic in $\PP^2$.
The first two possibilities are not realized and this means that $C'\to \widetilde{C'}$ is birational. Hence, by base chance, $R^0\pi_*(\I_{p,C'}(H)\otimes \E\otimes \O_{\uY}(h))$ has rank on $\widetilde{C'}$ equal to the rank of $\I_{p,C'}(H)\otimes \E\otimes \O_{\uY}(h)$ on $C'$, which is $2$.

We claim that $g$ in \eqref{eqn:defng} is injective. Assume, by contradiction that $g$ is zero. Hence, we have the following exact sequence
\begin{equation*}
0\to {\B_{-1}} \to \H^{0}(\Psi(\sigma^*F)) \to \H^{0}(\Psi(\sigma^*\I_C(H)))\to 0
\end{equation*}
and we have two cases depending on the behaviour morphism $f':=h\circ f:\B_2\to \H^{-1}(\Psi(\O_{\gamma}(H)))$.

\smallskip

\noindent{\emph{Case (a.1).}} If the map $f'$ is non-zero, then \eqref{eqn:fediag} yields the sequence
		\begin{equation*}\label{eqn:compn0}
			0 \to \H^{-1}(\Psi(\sigma^*\I_C(H)))\to \B_1 \stackrel{f_1}\to F_{\widetilde{C'}} \to \H^{0}(\Psi(\sigma^*\I_C(H)))\to 0,
		\end{equation*}
where $f_1$ could be either zero, or surjective, or have cokernel supported on points (see Step 1). If $\coker f_1=\H^{0}(\Psi(\sigma^*\I_C(H)))$ is supported at most in dimension $0$, then we have $0\neq \Hom^2(\B_1,\B_{-1})\hookrightarrow \Hom^2(\B_1,\H^0(\Psi(\sigma^*F)))$. So we get a contradiction with $\Psi(\sigma^* F)\in\langle \B_1\rangle^\perp$. Assume that $f_1$ and $g$ are both zero. Then $\H^{-1}(\Psi(\sigma^* F))=\B_1$ and we get a contradiction as $\Hom^0(\B_1,\H^{-1}(\Psi(\sigma^*F)))\neq 0$. Hence, $f_1$ is surjective and then, as in Step 1, $g$ is injective.

\smallskip

\noindent{\emph{Case (b.1).}} On the other hand, the map $f'$ could be zero, so $f$ would factor through $\B_2 \to \H^{-1}(\Psi(\I_{p,C'}(H)))$.
		In this case, we get a sequence
		\begin{equation*}\label{eqn:comp0}
			0 \to \H^{-1}(\Psi(\sigma^*\I_C(H)))\to \B_2 \stackrel{f_2}\to F_{\widetilde{C'}} \to K\to 0,
		\end{equation*}
		where $f_2$ could be either zero, or surjective, or have cokernel supported on points (see again Step 1).
		Moreover, we have
		\begin{equation}\label{eqn:comp02}
			0 \to K\to \H^{0}(\Psi(\sigma^*\I_C(H)))\to \H^{-1}(\Psi(\O_\gamma(H)))\to 0.
		\end{equation}
If $K$ is supported at most in dimension $0$, then
		\begin{equation*}
			\Ext^1(\B_1,\H^{0}(\Psi(\sigma^*\I_C(H))))=\Ext^1(\B_1,\H^{-1}(\Psi(\O_{\gamma}(H))))=0,
		\end{equation*}
		by \eqref{eqn:comp02} and \eqref{eqn:psigammaH}. So, again we have $0\neq \Hom^2(\B_1,\B_{-1})\hookrightarrow \Hom^2(\B_1,\H^0(\Psi(\sigma^*F)))$, contradicting $\Psi(\sigma^* F)\in\langle \B_1\rangle^\perp$.
		If $f_2$ and $g$ are zero, then $\H^{-1}(\Psi(\sigma^* F))=\B_2$ and we get a contradiction because $\Hom^0(\B_1,\H^{-1}(\Psi(\sigma^*F)))\neq 0$. As in the previous step, $f$ cannot be surjective whenever $g$ is trivial.
	
\smallskip

	Therefore, $g$ is injective and $\Psi(\sigma^*F)$ is a torsion sheaf with class $2[\B_1]-2[\B_0]$.

\bigskip

\noindent \textbf{Step 3:} Assume that the associated elliptic quintic $C$ intersects $l_0$ with multiplicity $m$ in a point.
	Since $C\cap l_0=\set{p}$ with multiplicity $m>1$, with the notation of Step 2, we have $\Psi(\sigma^*\O_C(H))= \Psi(\O_{C'\cup m\gamma}(H))$.
	Note that we have the following exact sequence on $\coh(\uY)$
	\begin{equation}\label{eqn:poplineM}
		0\to \I_{mp,C'}(H)\to \O_{C'\cup m\gamma}(H) \to \O_{m\gamma}(H)\to 0.
	\end{equation}
	Moreover $0\to \O_{(m-1)\gamma}(H)\to \O_{m\gamma}(H) \to \O_{\gamma}(H)\to 0$, so $\Psi(\O_{m\gamma}(H))$ is a successive extension of $\Psi(\O_{\gamma}(H))$.

	Then, we can also distinguish between two cases, depending on whether the morphism $\B_2\to \H^{-1}(\Psi(\O_{m\gamma}(H)))$ arising from \eqref{eqn:poplineM} and \eqref{eqn:defnf} is either non-zero (Case (a.2)), or $f$ factors through $\B_2 \to \H^{-1}(\Psi(\I_{p,C'}(H)))$ as in \eqref{eqn:defnf} (Case (b.2)).

	If we are in Case (b.2), exactly the same arguments as in Case (b.1) show that $\Psi(\sigma^*F)$ is a torsion sheaf with class $2[\B_1]-2[\B_0]$.
	So we can suppose that we are in Case (a.2) and we have the following diagram:
		\begin{equation*}
			\xymatrix{&&0\ar[d]&0\ar[d]\\
				&0\ar[r]&\H^{-1}(\Psi(\sigma^*\I_C(H)))\ar[r]\ar[d]& K'\ar[d]\\ % curved arrow 1
				&&\B_{2}\ar@{=}[r]\ar[d]^f& \B_2\ar[d] \\ %intersection point 1
				0\ar[r]&\H^{-1}(\Psi(\I_{p,C'}(H)))\ar[r]\ar@{=}[d]
				 %intersection point 2
				&\H^{-1}(\Psi(\O_{C'\cup \gamma}(H)))\ar[d]\ar[r]&\H^{-1}(\Psi(\O_{m\gamma}(H)))\ar[d]\ar[r]&0\\
				&\H^{-1}(\Psi(\I_{p,C'}(H)))\ar[r]& \H^{0}(\Psi(\sigma^*\I_C(H)))\ar[d]\ar[r] &L'\ar[d]\ar[r]&0\\
				&&0&0
					\hole
			}
		\end{equation*}

		By the Horseshoe Lemma we have the following exact sequence
		\begin{equation*}
			0\to \B_{1}^{\oplus m}\to     \B_2^{\oplus m} \to \H^{-1}(\Psi(\O_{m\gamma}(H)))\to 0.
		\end{equation*}
		So, we have that $K'\cong \B_1$ and $L'\cong \H^{-1}(\Psi(\O_{(m-1)\gamma}(H)))$ and we get the following
		\begin{equation*}
			0 \to \H^{-1}(\Psi(\sigma^*\I_C(H)))\to \B_1 \stackrel{f_3}\to F_{\widetilde{C'}} \to \H^{0}(\Psi(\sigma^*\I_C(H)))\to  \H^{-1}(\Psi(\O_{(m-1)\gamma}(H)))\to 0,
		\end{equation*}
		where $f_3$ could be either zero, or surjective, or have cokernel supported on points (see the argument in Step 1).

	\medskip

	We claim that $g$ in \eqref{eqn:defng} is injective.
	Running the same machinery as in the previous steps, we assume, by contradiction, that $g$ is zero.
	Hence, we have the following exact sequence
	\begin{equation*}
	0\to {\B_{-1}} \to \H^{0}(\Psi(\sigma^*F)) \to \H^{0}(\Psi(\sigma^*\I_C(H)))\to 0.
	\end{equation*}
	If $\coker f_3$ is supported at most in dimension $0$,
	\[
		\Ext^1(\B_1,\H^{0}(\Psi(\sigma^*\I_C(H)))=\Ext^1(\B_1,\H^{0}(\Psi(\O_{m\gamma}(H))))=0.
	\]
	So, we have $0\neq \Hom^2(\B_1,\B_{-1})\hookrightarrow \Hom^2(\B_1,\H^0(\Psi(\sigma^*F)))$, contradicting $\Psi(\sigma^* F)\in\langle \B_1\rangle^\perp$.
	Thus, it remains to deal with the case when $f_3$ and $g$ are zero. But then $\H^{-1}(\Psi(\sigma^* F))=\B_1$ and we get a contradiction because $\Hom^0(\B_1,\H^{-1}(\Psi(\sigma^*F)))\neq 0$.

If $f$ factors through $\B_2 \to \H^{-1}(\Psi(\I_{p,C'}(H)))$ and $K$ is supported at most in dimension 0, then
		\begin{equation*}
			\Ext^1(\B_1,\H^{0}(\Psi(\sigma^*\I_C(H))))=\Ext^1(\B_1,\H^{-1}(\Psi(\O_{\gamma}(H))))=0,
		\end{equation*}
		by \eqref{eqn:comp02} and \eqref{eqn:psigammaH}. So, again we have $0\neq \Hom^2(\B_1,\B_{-1})\hookrightarrow \Hom^2(\B_1,\H^0(\Psi(\sigma^*F)))$ and we get a contradiction with $\Psi(\sigma^* F)\in\langle \B_1\rangle^\perp$.
		As in the previous steps, if $f_2$ and $g$ are zero, then $\H^{-1}(\Psi(\sigma^* F))=\B_2$ and we get a contradiction because $\Hom^0(\B_1,\H^{-1}(\Psi(\sigma^*F)))\neq 0$.
	Therefore, $\Psi(\sigma^*F)$ is a torsion sheaf with class $2[\B_1]-2[\B_0]$.

\bigskip

\noindent \textbf{Step 4:} Assume that the associated elliptic quintic $C$ intersects $l_0$ in $s$ distinct points (possibly with multiplicity).
	Since $C\cap l_0=\set{p_1,\ldots, p_s}$, with the notation of Steps 2 and 3, we have $\Psi(\sigma^*\O_C(H))= \Psi(\O_{C'\cup m_1\gamma_1\cup\ldots \cup m_s\gamma_s}(H))$.
	Note that we have the following exact sequence in $\coh(\uY)$
	\begin{equation}\label{eqn:poplinemore}
		0\to \I_{m_1p_1\cup\ldots\cup m_sp_s,C'}(H)\to \O_{C'\cup m_1\gamma_1\cup\ldots \cup m_s\gamma_s}(H) \to \bigoplus_{i=1}^s\O_{m_i\gamma_i}(H)\to 0,
	\end{equation}
	since $\gamma_i$ are disjoint.
	Then, we can reduce to the previous steps.

\medskip

	From Step 1-4, we get that if $F$ is a stable Ulrich bundle of rank $2$, then $\Psi(\sigma^*F)$ is a torsion sheaf with class $2[\B_1]-2[\B_0]$.

\bigskip

\noindent \textbf{Step 5:} 	We can now show that $\Psi(\sigma^*F)$ is stable.
	Suppose it is not, and let $E\hookrightarrow \Psi(\sigma^*F)$ be a destabilizing $\B_0$-module.
	Then, $\ch(E)=(0,2,z)$ for some $z> -2$.
	Since $\Hom(\B_1,E)\hookrightarrow \Hom(\B_1,\Psi(\sigma^*F))=0$, then $\chi(\B_1,E)\leq 0$.
	By \eqref{eqn:chiB1F}, we have $z=-2$ and $E\in \langle \B_1\rangle^{\perp}$.
	By \cite[Thm.~4.1]{BMMS}, one gets that $E\cong\Xi_3(\I_l)$ for some line $l\in F(Y)\setminus\set{l_0}$ and $\I_l\hookrightarrow F$ contradicts the stability of $F$.
\end{proof}

	Denote by $\MMM_Y^{inst}$ the moduli space of semistable instanton sheaves.
	By \cite[Thm.~4.8]{Druel}, $\MMM_Y^{inst}$ is isomorphic to the blow-up $f\colon\MMM_Y^{inst}\to J(Y)$ of the intermediate Jacobian $J(Y)$ of $Y$ along (a translate of) $-F(Y)$.
	Here $F(Y)$ is the Fano surface of lines in $Y$.
	Recall that the Abel--Jacobi map establishes and isomorphism $\mathrm{Alb}(F(Y))\stackrel{\sim}\to J(Y)$.
	Moreover, the Albanese morphism provides an embedding $F(Y)\hookrightarrow \mathrm{Alb}(F(Y))$ defined after picking a special point, in our case $l_0$ (see \cite{CG}).
	Recall that, the closure of the locus of stable non-locally free instanton sheaves (case (2) in Remark \ref{rmk:nlf_inst}) forms the exceptional divisor of $f$.
	Stable non-locally free instanton sheaves are associated to a smooth conic inside $Y$ via the Serre construction and they are sent to the residual line of the conic.
	
	We denote by $\overline{F(Y)}$ the strict transform of $F(Y)$ under $f$.
	Since we have chosen $l_0$ general (i.e., such that for any other line $l$ meeting $l_0$, the plane containing them intersects the cubic in three distinct lines), then $F(Y)\cap (-F(Y))$  is the Abel--Prym curve $C_{l_0}\subset J(Y)$ consisting of all lines inside $Y$ that intersect $l_0$ (e.g.~\cite[Sect. 5]{LN}).
	Note that $\overline{F(Y)}$ parametrizes properly semistable instanton sheaves that fall under case (3) in Remark \ref{rmk:nlf_inst} and are extensions of $\I_{l_0}$ and $\I_{l}$, for $l$ a line in $Y$ (possibly equal to $l_0$).
	Indeed, semistable instanton sheaves under case (3) have second Chern class $c_2(E)=l+l_0$.

	Therefore, $\overline{F(Y)}$ intersects the divisor contracted by $f$, in the locus where $E$ is semistable and it is the extension of $\I_{l_0}$ and $\I_{l}$, such that $l\cap l_0\neq \emptyset$.
	From the point of view of the conics, this corresponds to the case where the conic over $l$ degenerates to $l_0\cup l'$, where $l,l',l_0$ are coplanar and in general position.
	
\begin{thm}\label{thm:d=2}
	The moduli space $\MMM_2$ is the blow-up of $\MMM_Y^{inst}$ along $\overline{F(Y)}$.
\end{thm}

\begin{proof}
	Let $E$ be a semistable instanton sheaf.
	We claim that if $\Xi_3(E)\in \coh(\PP^2,\B_0)$, then $\Xi_3(E)\in \MMM_2$ (i.e., it is semistable) and, by Lemma \ref{lem:stabM2}, $\Xi_3(E)\in \MMM^{\sigma_m}(\PP^2,\B_0;w)$ for all $m>\frac{1}{4}$.
	Combining Remark \ref{rmk:nlf_inst} and \cite[Ex.~2.4, Ex.~2.11, Step 5 in Prop.~3.3]{BMMS}, we can distinguish three cases where $\Xi_3(E)\in \coh(\PP^2,\B_0)$.
	
	One possibility is that $E$ is a stable instanton bundle.
	In this case, $\Xi_3(E)\in \MMM_2$ follows from Lemma \ref{lem:inst>B0stab}.

	Another possibility is that $E$ is a stable instanton sheaf which is not locally free.
	In that case $E$ can be associated to a smooth conic via the Serre construction.
	If the conic does not intersect the line of projection then $\Xi_3(E)\in \MMM_2$ follows from Example \ref{ex:objMd}.
	If the conic intersects the line of projection in one point or two points (even tangentially), then the same computations as in Steps 2, 3, and 4 of the proof of Lemma \ref{lem:inst>B0stab} show again that $\Xi_3(E)\in \coh(\PP^2,\B_0)$.
	By Step 5 of the proof of Lemma \ref{lem:inst>B0stab}, $\Xi_3(E)$ is stable, so in $\MMM_2$.
	
	Finally, the last possibility is that $E$ is a properly semistable sheaf and the 2 JH-factors are $\I_{l}$ and $\I_{l'}$, where eventually $l=l'$, but in any case $l,l'\neq l_0$.
	Then $\Xi_3(E)\in \MMM_2$ follows from a direct computation based on the fact that $\Xi_3(\I_{l})$ and $\Xi_3(\I_{l'})$ are in $\MMM_1^s$ (see Lemma \ref{lem:stabM1}).
	
	Hence, by \cite[Ex.~2.11]{BMMS}, the only cases in which $\Xi_3(E)\not\in \coh(\PP^2,\B_0)$ appear when $E$ is a properly semistable sheaf and $\I_{l_0}$ is a JH-factor. Indeed, this is the only case where $\Xi_3(E)\not\in \MMM_2$ and we need to push our analysis a bit further.

	When $\I_{l_0}$ and $\I_{l}$ with $l\neq l_0$ are the JH-factors of $E$, then $\hom(\Xi_3(E),\B_1)=1$.
	Hence, the HN-filtration of $\Xi_3(E)$ for $m > m_0=\frac{\sqrt{5}}{8}$ is
	\begin{eqnarray*}
		\B_0[1] \subset C[1] \subset \Xi_3(E),
	\end{eqnarray*}
	where $0\to C[1] \to \Xi_3(E) \to \B_1\to 0$ and $0\to \B_0[1] \to C[1] \to \Xi_3(\I_l)\to 0$ are exact sequences in the abelian category $\cat{A}$ which is the heart of the bounded $t$-structure in the stability condition in Lemma \ref{lem:exstab}.
If the two JH-factors of $E$ are both isomorphic to $\I_{l_0}$, then $\hom(\Xi_3(E),\B_1)=2$.
	Hence, the HN-filtration of $\Xi_3(E)$ for $m > m_0$ is
	\begin{eqnarray*}
		\B_0[1] \subset \B_0^{\oplus 2}[1] \subset C[1] \subset \Xi_3(E),
	\end{eqnarray*}
	where $0\to \B_0^{\oplus 2}[1] \to C[1] \to \B_1\to 0$ and $0\to C[1] \to \Xi_3(E) \to \B_1\to 0$ are exact sequences in the abelian category $\cat{A}$.

	In both cases, this means that $\Xi_3(E)$ is $\sigma_{m_0}$-semistable.
	As a consequence, up to choosing $\eps$ small enough, $\Xi_3(E)$ is $\sigma_{m}$-semistable, for all $m \in (m_0-\eps , m_0 )$.
	Indeed, if not, by \cite[Prop.~9.3]{Br1}, the HN-factors of $\Xi_3(E)$ in the stability condition $\sigma_m$,
	for $m \in (m_0-\eps, m_0 )$, would survive in the stability condition $\sigma_{m_0}$.
	This would contradict the $\sigma_{m_0}$-semistability of $\Xi_3(E)$.

	Since the quotient $\Xi_3(E)\to \B_1$ $\sigma_m$-destabilizes $\Xi_3(E)$, for $m>m_0$, we have that
	\begin{align*}
		\Xi_3(E)&\not\in \MMM^{\sigma_m}(\PP^2,\B_0;w)\qquad \text{ for }m>m_0:=\frac{\sqrt{5}}{8},\\
		\Xi_3(E)&\in \MMM^{\sigma_m}(\PP^2,\B_0;w)\qquad \text{ for }m\in (m_0-\eps,m_0]\text{ and }\eps>0\text{ small enough,}
	\end{align*}
where $w:=2[\B_1]-2[\B_0]$.

\smallskip
	
	We claim that $\MMM_Y^{inst}=\MMM^{\sigma_{m}}(\PP^2,\B_0;w)$, for all $m\in (m_0-\eps,m_0]$ and $\eps>0$ sufficiently small.
	More precisely, we need to show that, for $m \in (m_0-\eps , m_0 )$, the objects $\Xi_3(E)$ are the only $\sigma_m$-semistable objects in $\cat{A}$ with class $w$, when $E$ is a semistable instanton sheaf.
	First observe that, if $G \in\cat{A}$ is a $\sigma_m$-semistable object, for some $m \in (m_0-\eps,m_0)$ and with class $w$, then $G\in \langle\B_1\rangle^{\perp}$.
	Indeed, if $\hom(\B_1,G)\neq 0$, then the image $T:=\mathrm{Im}_{\cat{A}}(\B_1\to G)$ destabilizes $G$, for $m\in (m_0-\eps,m_0)$ and $\eps>0$ small enough, since $\mathrm{Re}\,( Z_{m_0}(T))>0$.
	%On the other hand, if $\hom^1(\B_1,G)=\hom(G,\B_0[1])\neq 0$, then the surjection $G \twoheadrightarrow \B_0[1]$ in $\cat{A}$ destabilizes $G$, for $m\in (m_0-\eps,m_0)$ and $\eps>0$ small enough.
	%In both cases, this contradicts the $\sigma_m$-semistability of $G$.

	By \cite[Prop.~9.3]{Br1}, up to replacing $\eps$, we can assume that all such objects $G$ are $\sigma_{m_0}$-semistable.
	By Lemma \ref{lem:keepstab}, we have two possibilities: either $G$ is $\sigma_{m}$-semistable for all $m\geq m_0$, or $G$ is properly $\sigma_{m_0}$-semistable and destabilizes for all $m>m_0$.
	In the first case, $G$ is a (semi)stable element of $\NNN_2$ by Lemma \ref{lem:smstab-stab} and the discussion before.
	Thus, by the proof of Proposition \ref{prop:exACM}, $\Xi_3^{-1}(G)$ is either a balanced ACM bundle of rank $2$ (i.e., an instanton bundle) or as in case (2) of Remark \ref{rmk:nlf_inst}.

	If $G$ destabilizes for all $m>m_0$, then $G$ needs to be in cases (a) and (b) of Lemma \ref{lem:countext}.
	Since $G\in \gen{\B_1}^\perp$, Lemma \ref{lem:countext} tell us that $G\cong \Xi_3(E)$ where $E$ is a properly semistable sheaf with $\I_{l_0}$ as a JH-factor.
	
	\smallskip
	
	Having proved this, we are ready to show that $\MMM_2$ is the blow-up of $\MMM_Y^{inst}$ along $\overline{F(Y)}$.
	In view of Corollary \ref{cor:M2vsMsig}, one has to study the objects $F\in\MMM^{\sigma_{m}}(\PP^2,\B_0;w)=\MMM_2$, for all $m>m_0$ and $w=2[\B_1]-2[\B_0]$, which become $\sigma_{m_0}$-semistable with JH-factors as in cases (c.iii) and (c.iv) of Lemma \ref{lem:stabM2}.
	Indeed, by Lemma \ref{lem:stabM2}, these are the only objects that could be contracted.
	By Remark \ref{rmk:fibcontr}, the ones falling in case (a.i) of Lemma \ref{lem:countext} get contracted to the S-equivalence classes of the instanton sheaves which are extensions of $\I_{l_0}$ and $\I_l$ ($l\neq l_0$).
	For the same reason, the ones in case in case (b.i) of Lemma \ref{lem:countext} are contracted to the S-equivalence class of the instanton sheaves which are extensions of $\I_{l_0}$ with itself.
	Moreover, again by Lemma \ref{lem:countext}, each contracted fiber is $\PP^2$ and the birational map $\MMM_2\to \MMM_Y^{inst}$ is a well-defined morphism.
	
	Applying \cite[Thm.~2]{Luo}, we conclude that $\MMM_2$ is isomorphic to the blow-up of $\MMM_Y^{inst}$ at $\overline{F(Y)}$.
\end{proof}

As a corollary of the previous proof we get the following result which is of interest in itself.
\begin{cor} \label{cor:MinstvsMsig}
	Let $w=2[\B_1]-2[\B_0]$ and $m_0=\frac{\sqrt{5}}{8}$.
	Then $\MMM_Y^{inst}=\MMM^{\sigma_{m}}(\PP^2,\B_0;w)$, for all $m\in (m_0-\eps,m_0]$ and $\eps>0$ sufficiently small.
\end{cor}

%%%%%%%%%%%%%%%%%%%%%%%%

\medskip

{\small\noindent{\bf Acknowledgements.}
It is a pleasure to thank Nick Addington, Asher Auel, Marcello Bernardara, Robin Hartshorne, Daniel Huybrechts, Nathan Ilten, Sukhendu Mehrotra, Nicolas Perrin, Antonio Rapagnetta, Pawel Sosna, and Olivier Wittenberg for very useful conversations and comments.
We are also very grateful to the referee for pointing out several inaccuracies and a mistake in the statement of an early version of Theorem \ref{thm:mainB0}.
Parts of this paper were written while the three authors were visiting the University of Utah, the University of Bonn, the Ohio State University, and the University of Barcelona. The warm hospitality from these institutions is gratefully acknowledged.
}

%%%%%%%%%%%%%%%%%%%%%%%%

\end{document}